\documentclass[paper,onefignum,onetabnum]{siamart171218}

\usepackage{lipsum}
\usepackage{amsfonts}
\usepackage{graphicx}
\usepackage{epstopdf}
\usepackage{algorithmic}

\usepackage{array}
\usepackage{tikz}
\usepackage{xcolor}
\usepackage{hyperref}
\usepackage{enumitem}
\usepackage{amssymb}
\usepackage{subcaption}

\ifpdf
  \DeclareGraphicsExtensions{.eps,.pdf,.png,.jpg}
\else
  \DeclareGraphicsExtensions{.eps}
\fi

\newsiamremark{remark}{Remark}
\newsiamremark{hypothesis}{Hypothesis}
\crefname{hypothesis}{Hypothesis}{Hypotheses}
\newsiamthm{claim}{Claim}

\headers{Breaking the Odd Cycle}{E. Bellesso, I. M. Bomze, and F. Rinaldi}
\title{Breaking the Odd Cycle: Regularization and First-Order Optimization for the Maximum \(s\)-Plex Problem}

\author{
Eleonora Bellesso\footnotemark[1]\thanks{Dipartimento di Matematica ``Tullio Levi-Civita'',
  Università di Padova, Padova PD, Italy
  (\email{eleonora.bellesso@studenti.unipd.it},
   \email{rinaldi@math.unipd.it}).}
\and
Immanuel M. Bomze\footnotemark[2]\thanks{Fakult\"at f\"ur Mathematik \& ds:UniVie, Universität Wien,
  Wien, Austria
  (\email{immanuel.bomze@univie.ac.at}).}
\and
Francesco Rinaldi\footnotemark[1]
}

\usepackage{amsopn}

\makeatletter
\newcommand*{\addFileDependency}[1]{% argument=file name and extension
  \typeout{(#1)}% latexmk will find this if $recorder=0 (however, in that case, it will ignore #1 if it is a .aux or .pdf file etc and it exists! if it doesn't exist, it will appear in the list of dependents regardless)
  \@addtofilelist{#1}% if you want it to appear in \listfiles, not really necessary and latexmk doesn't use this
  \IfFileExists{#1}{}{\typeout{No file #1.}}% latexmk will find this message if #1 doesn't exist (yet)
}
\makeatother

\newcommand*{\myexternaldocument}[1]{%
    \externaldocument{#1}%
    \addFileDependency{#1.tex}%
    \addFileDependency{#1.aux}%
}
\ifpdf
\hypersetup{
pdftitle={Breaking the Odd Cycle: Regularization and  
First-Order Optimization for the Maximum $s$-Plex Problem},
  pdfauthor={E. Bellesso, I. M. Bomze, and F. Rinaldi}
}
\fi

\myexternaldocument{ex_supplement}

\begin{document}

\maketitle

% REQUIRED
\begin{abstract}
    Building on a new regularized continuous formulation  which avoids spurious solutions of the  maximum $s$-plex problem, we develop a tailored block-coordinate first-order method that combines away-step Frank-Wolfe updates on the standard simplex with linear optimization over the fractional $b$-matching polytope. We establish finite stabilization of the auxiliary variables, convergence to stationary points, and finite support identification under suitable conditions. Numerical experiments on standard benchmark graphs show that, also in practice, regularization avoids the support-identification failures observed with the original cubic formulation and enables the proposed method to quickly find high-quality $s$-plexes.
  
  % In Social Network Analysis, relationships between members are encoded in an undirected graph where vertices represent the members of the network, and edges indicate the existence of a relationship between them. This paper analyzes the $s$-plex clique relaxation model and develops a general regularization-based continuous optimization framework for the maximum $s$-plex problem (M$s$PP). In particular, we consider a broad class of regularization terms that can be included in the associated Motzkin-Straus cubic formulation, and we develop conditions that guarantee the equivalence between the continuous regularized problem and the original one in both a global and a local sense. Some regularization functions are also presented which meet the detected conditions. We further analyze, from a computational point of view, a tailored variant of the Frank-Wolfe algorithm that allows us to quickly find $s$-plexes within a given graph. The reported numerical results clearly show the effectiveness of the proposed approach, which is closely connected to its support identification properties, and make it a very competitive algorithm to search for maximum $s$-plexes in practical applications.
\end{abstract}

% REQUIRED

% REQUIRED

% The introduction introduces the context and summarizes the manuscript. It is importantly to clearly state the contributions of this piece of work.
\section{Introduction}
The Motzkin-Straus (MS) formulation \cite{M-S} is one of the most influential examples of how a combinatorial graph problem can be approached through continuous optimization. By maximizing a quadratic function over the standard simplex, it recovers the clique number of a graph from the resulting optimal value. Besides its theoretical significance, this connection has inspired several continuous approaches to the maximum clique problem (MCP); see, for example, \cite{evolution,Hungerford}. The correspondence, however, is incomplete at the level of solutions: maximizers of the MS formulation need not have clique support. These so-called spurious solutions are particularly problematic for algorithms that recover a clique from the support of a continuous solution and have motivated the development of suitable regularizations of the original formulation \cite{feasible_and_infeasible}.

The influence of the MS formulation, however, extends beyond cliques. In many applications, particularly in community detection, requiring every pair of selected vertices to be adjacent is overly restrictive, since meaningful communities are generally dense but not complete. This observation has led to several clique-relaxation models, which describe cohesive subgraphs while allowing controlled deviations from complete connectivity; see \cite{discussion,clique_relaxations} and the references therein. In this paper, we focus on the $s$-plex model, introduced by Seidman and Foster \cite{first_s-plex} in the context of social network analysis. A vertex set is an $s$-plex if each of its vertices has at most \(s-1\) non-neighbors within the set. Thus, \(1\)-plexes are precisely cliques, whereas values of \(s>1\) permit progressively weaker connectivity requirements. This model should not be confused with the $s$-defective clique relaxation, introduced by Yu et al. \cite{yu2006predicting} in the context of protein-protein interaction networks, which allows up to \(s\) missing edges in the vertex set. These two models have often been studied together due to shared properties, such as heredity, despite the fact that the internal structures they can generate are fundamentally different and that, in general, no process exists that allows one structure to be transformed into the other. In \cite{cubic}, two cubic MS-type formulations for the maximum $s$-plex problem (M$s$PP) and the maximum $s$-defective clique problem (M$s$DCP) were proposed separately. For both cases, in addition to a vertex vector belonging to the standard simplex, the formulations employ variables indexed by the graph's missing edges. These variables encode feasible sets of missing edges, subject to the constraint that at most $s-1$ of them can be incident to any single vertex for the $s$-plex case, whereas at most $s$ missing edges can be selected for the $s$-defective clique model. In both formulations, the global optimal value of the cubic problem allows one to derive the maximum cardinality of the sought relaxed structure. However, as is the case with the original MS formulation, this characterization of the optimal value does not automatically extend to the structure of the maximizers: the cubic formulation can indeed admit spurious solutions whose vertex supports are not clique relaxations.

%A cubic MS-type formulation for the maximum $s$-plex problem (M$s$PP) was proposed in \cite{cubic}. In addition to a vertex vector belonging to the standard simplex, the formulation employs variables indexed by the missing edges of the graph. These variables encode admissible sets of missing edges, subject to the constraint that at most \(s-1\) of them be incident to any vertex. The global optimal value of the resulting cubic problem determines the maximum $s$-plex cardinality. As in the original MS formulation, however, this optimal-value characterization does not extend automatically to the structure of its maximizers: the cubic formulation may admit spurious maximizers whose vertex supports are not $s$-plexes.

For the $s$-plex formulation, eliminating these solutions is technically more involved than in both the clique and the $s$-defective clique cases. For the MCP, there are no missing-edge variables involved, so the regularization acts only on the vertex variables over the standard simplex; see \cite{Hungerford}. The $s$-defective clique formulation does introduce missing-edge variables, and its continuous relaxation contains fractional feasible points. Nevertheless, its feasible region is an integral polytope for integer $s$, meaning that all its extreme points are binary. Consequently, strict convexity with respect to the missing-edge variables is sufficient to ensure that every local maximizer of the corresponding fixed-vertex problem is binary \cite{s-defective}. This argument does not extend directly to the $s$-plex formulation, where the separate degree constraints imposed at the vertices define a fractional $b$-matching polytope. Its extreme points may themselves be fractional; more precisely, they are half-integral, with their one-half components supported on vertex-disjoint odd cycles. Strict convexity therefore guarantees only that a local maximizer is an extreme point, not that its missing-edge component is binary. The purpose of the $y$-regularization is therefore to break the odd cycles supporting fractional extreme points, in the
sense of making every such half-integral configuration locally improvable while leaving binary missing-edge selections unaffected.
Achieving this requires additional conditions on the regularization and constitutes the main technical distinction of the $s$-plex setting.

It is hence easy to understand why a framework for the $s$-plex problem cannot be obtained through a direct adaptation of the regularizations developed for the clique and $s$-defective clique problems. Such a regularizer must simultaneously control the geometry of the objective over the simplex and eliminate half-integral configurations in the missing-edge block, while leaving binary missing-edge selections unpenalized. As we will see, by deriving regularization conditions that achieve both goals through an appropriate exploitation of the problem’s distinctive polyhedral structure, we eliminate the spurious maximizers of the cubic formulation and ensure that every local maximizer encodes an $s$-plex and every global maximizer encodes a maximum $s$-plex. \\
The main contributions of the paper can be summarized as follows:
\begin{itemize}
\iffalse
\item We introduce a general class of regularized continuous formulations for the M$s$PP and derive conditions on the regularization terms that ensure the desired correspondence between their maximizers and $s$-plexes. We first address the missing-edge component. By exploiting the half-integrality of the fractional $b$-matching polytope together with the downward-closed structure of its degree constraints, we show that the regularization rules out local maximizers with fractional missing-edge components. Since the one-half components of every nonbinary extreme point are supported on vertex-disjoint odd cycles, this prevents all such fractional odd-cycle configurations from occurring at local maximizers. We then characterize the local and global maximizers of the regularized problem. In particular, every local maximizer has a binary missing-edge component and a vertex component given by the normalized characteristic vector of an $s$-plex, whereas global maximizers correspond to maximum $s$-plexes.
\fi  % seems a repeat from below (or vice versa, pls. choose)
    \item We introduce a general class of regularized continuous formulations for the M$s$PP and derive conditions on the regularization terms that exclude local maximizers with fractional missing-edge components and vertex supports that do not correspond to an $s$-plex. We first focus on the binary complementary structure. Taking into account the half-integrality of the fractional $b$-matching polytope, we show that the regularization rules out local maximizers with fractional missing-edge components. Since the one-half components of every non-binary extreme point are supported on vertex-disjoint odd cycles, this result excludes all such fractional odd-cycle configurations. We then characterize the local and global maximizers of the regularized problem. In particular, every local maximizer has a binary missing-edge component and a vertex component given by the normalized characteristic vector of an $s$-plex, whereas global maximizers correspond to maximum $s$-plexes.
    
    \item We provide and analyze several families of regularization functions satisfying the proposed conditions.
    
    \item We develop a tailored block-coordinate first-order method that combines away-step Frank-Wolfe updates on the standard simplex with linear optimization over the fractional $b$-matching polytope. We establish finite stabilization of the missing-edge variables, convergence to stationary points, and finite support identification under suitable conditions.
    
    \item We evaluate the proposed method on a standard family of benchmark graphs. The reported numerical results show that regularization avoids the support-identification failures observed with the original cubic formulation and enables the proposed method to quickly find high-quality $s$-plexes.
\end{itemize}

Frank-Wolfe methods \cite{frank1956algorithm} are particularly attractive in this setting because they avoid projections and instead rely on linear optimization over the feasible region. This projection-free feature is crucial when dealing with geometrically complex polytopes, such as the fractional $b$-matching polytope considered here. Indeed, its Euclidean projection has no simple closed-form expression and would require solving a non-trivial constrained quadratic problem at every iteration, making projection-based methods computationally impractical. By contrast, the Frank-Wolfe linear optimization oracle can directly exploit the underlying $b$-matching structure. Moreover, such projection-free algorithms are able to generate sparse iterates and identify the support of a limiting solution. This aligns well with the combinatorial interpretation of the simplex variables. The away-step mechanism is especially useful for removing incorrectly selected vertices, while the block-coordinate structure makes it possible to treat the simplex and fractional $b$-matching components using procedures tailored to their respective geometries.

The remainder of the paper is organized as follows. In section~\ref{section_2}, we introduce notation, review the cubic formulation, and recall some preliminary results on the fractional $b$-matching polytope. Section~\ref{section_3} develops the general regularization framework and characterizes the local and global maximizers of the resulting formulation. Section~\ref{section_4} presents and analyzes several families of regularization functions satisfying the proposed conditions. In section~\ref{section_5}, we present the block-coordinate first-order method and establish its finite-time stabilization, convergence, and support-identification properties. Section~\ref{section_6} reports the computational study, including the comparison between the original and regularized formulations.  Section~\ref{section_7} concludes the paper.

The implementation of the proposed method, together with the benchmark instances used in the computational study, is publicly available at the following repository: \url{https://github.com/eleonora-bellesso/s-plex_FW}.

\section{Notation and preliminaries}\label{section_2}
Let $G=(V,E)$ be a simple undirected graph with vertex set $V=\{1,\dots,n\}$ and edge set $E\subseteq \binom{V}{2}$ consisting of $m=|E|$ unordered pairs of adjacent vertices. Similarly, denote by $\bar{G}=(V,\bar{E})$ the complementary graph of $G$, where $\bar{E}=\binom{V} {2}\setminus E$, and $\bar{m}=|\bar{E}|$ is the number of missing edges in $G$. For any subset of vertices $S\subseteq V$, we denote by $G[S]=(S,E(S))$ the subgraph induced by $S$ in $G$, where $E(S)=E\cap\binom{S}{2}$. The notation adopted throughout this article largely follows that introduced in \cite{cubic}.

\begin{definition}
Given a positive integer $s$, a subset of vertices $S\subseteq V$ is an $s$-plex if every vertex in $S$ is adjacent to at least $|S|-s$ other vertices of $S$. Equivalently, every vertex in $S$ has at most $s-1$ non-neighbors in $S$.
\end{definition}

The maximum $s$-plex problem aims to find an $s$-plex of maximum cardinality in $G$. This cardinality is denoted by $w_s^p(G)$ and is called the $s$-plex number of $G$.

In \cite{cubic}, the authors showed that the quadratic MS formulation for the MCP can be extended to a continuous cubic formulation for the M$s$PP by introducing a vector $y$ of auxiliary variables indexed by the missing edges of $G$. These variables represent the ``fake edges'' that may be added to the graph in order to recover a clique structure from the $s$-plex relaxation. 

In the following, let $A_G=[a_{ij}]_{i,j\in V}$ be the $n\times n$ adjacency matrix of $G$ such that $a_{ij} = 1$ if $\{i,j\}\in E$ and $0$ otherwise, and $N(v)$ the neighborhood of any vertex $v\in V$ which indeed corresponds to the set of adjacent vertices to $v$ in $G$. We instead indicate with $B_{\bar{G}}=[b_{vc}]_{v\in V,c\in\bar{E}}$ the $n\times\bar{m}$ incidence matrix of the complementary graph $\bar{G}$, such that $b_{vc}=1$ if $v$ is an endpoint of the missing edge $c$, so $v\in c$; and $0$ otherwise. 

Let $\Delta_n=\{x\in\mathbb{R}^n\ |\ x\geq0,\ e^\top x=1\}$ be the standard simplex in $\mathbb{R}^n$. Here $e$ is always a vector with all components equal to $1$ and the dimension clear from the context. Similarly, we denote by $\mathbb{I}_{n}$ the $n\times n$ identity matrix and by $e_i\in\mathbb{R}^n$ its $i$-th column. For any vector $z$ whose components are indexed by a finite set $I$, we denote its support by $\operatorname{supp}(z):=\{i\in I:z_i\neq 0\}$.
For instance, for any nonempty $S\subseteq V$, we can define the (normalized) characteristic vector of $S$ as
$$x(S):=\frac 1{|S|}\,\sum_{i\in S}e_i\in \Delta_n$$ 
which is the barycenter of the face $\Delta(S):=\{x\in \Delta_n\ |\  \operatorname{supp}(x)\subseteq S\}$ whose support is $\operatorname{supp}(x(S)) = S$.

For each missing edge $\{i,j\}\in\bar{E}$, we formally introduce a binary variable $y_{ij}$. For any fixed positive integer $s\geq1$, we define
\begin{align}
    P_s(G) = \{y\in \{0,1\}^{\bar{m}}\ |\ B_{\bar{G}} y \leq (s - 1)e\}
\end{align}
as the set of incidence vectors of admissible missing-edge sets that may be added to the original graph $G$ in order to complete an $s$-plex to a clique, and its continuous relaxation as
\begin{align}
    P'_s(G) = \{y\in [0,1]^{\bar{m}}\ |\ B_{\bar{G}} y \leq (s - 1)e\}.
\end{align}
For any $y\in P'_s(G)$ we define $A_{\bar{G}}(y)$ as the matrix with entry in the $i$-th row and the $j$-th column given by $y_{ij}\in[0,1]$ if the edge $\{i,j\}\in\bar{E}$ and $0$ otherwise. That is, $A_{\bar{G}}(y)$ is obtained from the adjacency matrix $A_{\bar{G}}$ of the complementary graph $\bar{G}$ by replacing each non-zero entry with the corresponding $y_{ij}$ value. 
Using the notation introduced above, we can now recall the main continuous characterization of the $s$-plex number established in \cite{cubic}. The result shows that the integrality constraints on the missing-edge variables can be relaxed without changing the global optimal value, thereby extending the classical MS formulation from cliques to $s$-plexes.

\begin{theorem}[{\cite[Theorem~2]{cubic}}] \label{MS_splex_theorem} The $s$-plex number $w_s^p(G)$ of a graph $G$ can be expressed in terms of the optimal value of the following continuous cubic problem:
\begin{align}\label{MS_splex_formulation}
    1 - \frac{1}{w_s^p(G)} = \max_{(x,y)\in\Delta_n\times P'_s(G)} f_G(x,y)
\end{align}
where the objective function is given by 
\begin{align}\label{MS_splex_objective_function}
    f_G(x,y) = x^\top (A_G + A_{\bar{G}} (y))x\ .
\end{align}
\end{theorem}
We observe that the objective function \eqref{MS_splex_objective_function}
is a polynomial of total degree three in the variables $(x,y)$. More precisely, it is
quadratic with respect to the vertex variables~$x$ and affine with respect
to the missing-edge variables $y$.
Moreover, the projection of the feasible set onto the $x$ variables coincides with the standard simplex $\Delta_n$ as in the original MS formulation. Indeed, for the specific value of $s=1$, the feasible region for the $y$ variable coincides with the single zero vector, $P'_1(G)=\{\underline{0}\}$, and also the induced adjacency matrix is simply given by the zero matrix, $A_{\bar{G}}(y)=\underline{0}$. This finally means that no missing edges can be included in the computation, and so the above problem~(\ref{MS_splex_formulation}) trivially coincides with the original MS formulation for the MCP, as expected.

It is also important to note that 
establishing the correctness of formulation (\ref{MS_splex_formulation}) is rather non-trivial. In particular, the authors in \cite{cubic} exploited a classical result concerning the half-integrality property of the fractional $b$-matching polytope.
Since this property will play a central role in our subsequent analysis, we decided to briefly review some fundamental concepts and results concerning the canonical $b$-matching problem. We begin by formally defining this structure.
\begin{definition}
     Given an undirected graph $G=(V,E)$ and a non-negative integer vector $b\in \mathbb{Z}_n^+$, a subset of edges $M\subseteq E$ is said to be a $b$-matching in $G$ if no vertex $i\in V$ is incident to more than $b_i$ edges in $M$.
\end{definition}
Moreover, given non-negative weights $w_{ij}$ associated with each edge $\{i,j\}\in E$, the maximum weight $b$-matching problem (MW$b$MP) aims to find a $b$-matching in $G$ with the largest total possible weight. In particular, an integer linear program can be formulated for this combinatorial optimization problem, where for any edge $\{i,j\}\in E$ a binary variable $z_{ij}$ is considered, indicating whether the corresponding edge is included or not in the $b$-matching structure.
\begin{align}\label{max_WbMP}
    \begin{split}
        \max \quad & \sum_{\{i,j\}\in E} w_{ij}z_{ij}  \\
    s.t. \quad & \sum_{j\in N(i)} z_{ij} \leq b_i,\ \forall\ i\in V  \\
    \quad & z_{ij}\in\{0,1\},\ \forall\ \{i,j\}\in E\, .
    \end{split}
\end{align}
The linear relaxation of problem~(\ref{max_WbMP}) reads
\begin{align}\label{LR}
    \begin{split}
        \max \quad & \sum_{\{i,j\}\in E} w_{ij}z_{ij}  \\
    s.t. \quad & \sum_{j\in N(i)} z_{ij} \leq b_i,\ \forall\ i\in V  \\
    \quad & z_{ij}\in [0,1],\ \forall\ \{i,j\}\in E\, .
    \end{split}
\end{align}
Clearly, problem~\eqref{LR} 
%(LR) of the above program  
may have a fractional optimal solution and, in general, a higher optimal value. Consider, for example, the case where $b=e$, the unit vector, and $w_{ij}=1$ for every edge $\{i,j\}\in E$, so the standard maximum matching problem, over the undirected graph $G=C_5$, so the simple cycle of $5$ vertices. Then the optimal solution to problem~\eqref{LR} is obtained by setting $z_{ij}=\frac{1}{2}$ for every edge $\{i,j\}\in E(C_5)$. This observation finally leads to the following well-known result on the fractional $b$-matching polytope~{\cite[pp. 44-46]{johnson1965programming}}.
\begin{lemma}\label{lemma_fractional_b-matching_polytope}
    Let $P$ be the feasible region of problem~\eqref{LR}. Any extreme point $\hat{z}$ of $P$ is half-integral, i.e., $\hat{z}_{ij}\in\{0,\frac{1}{2},1\},\ \forall\ \{i,j\}\in E$. Moreover, the edge-induced subgraph $G\langle F\rangle$, where $F=\{\{i,j\}\in E\ |\ \hat{z}_{ij}=\frac{1}{2}\}$, consists of disjoint odd cycles.
\end{lemma}
Here, the edge-induced subgraph $G\langle F \rangle=(V_F,F)$ has $F$ as its edge set, while its vertex set $V_F\subseteq V$ includes all vertices incident to the edges of $F$. 

% \IB{To begin with, we need a regularization procedure for the $y$-variables. We employ $\Psi$ with the following...
% \begin{lemma}\label{newlem}
% \end{lemma}
% %For any $x^*\in \Delta_n$, 
% Let $g(y)=  \ell(y)  + \Psi(y)$. Then any local maximizer of $g$ is a vertex of $P_s'(G)$.
% }

We now return to the MS formulation~\eqref{MS_splex_formulation} for the M$s$PP. As shown in the proof of Theorem~\ref{MS_splex_theorem}, fixing the $x$-component at some $x^*\in\Delta_n$ yields the following maximization problem in the missing-edge variables:
\begin{align}\label{LR_MWBMP_fix_x*}
    \begin{split}
        \max \quad & 2\sum_{\{i,j\}\in E} x^*_ix^*_j
        +2\sum_{\{i,j\}\in\bar{E}}x^*_iy_{ij}x^*_j\\
        \mathrm{s.t.}\quad
        & B_{\bar{G}}y\leq(s-1)e,\qquad
        y\in[0,1]^{\bar{m}}.
    \end{split}
\end{align}
Since the first term of the objective does not depend on $y$, problem~\eqref{LR_MWBMP_fix_x*} coincides, up to an additive constant, with the linear relaxation of the MW$b$MP~\eqref{max_WbMP} on the complement graph $\bar{G}$, with $b=(s-1)e$ and weights $w^*_{ij}=2x^*_ix^*_j$ for every $\{i,j\}\in\bar{E}$.

The proof of Theorem~\ref{MS_splex_theorem} exploits this connection to show that, starting from an arbitrary global maximizer of formulation~\eqref{MS_splex_formulation}, one can construct another global maximizer $(x^*,y^*)$ with the same objective value such that $y^*$ is binary and $\operatorname{supp}(x^*)$ is a maximum $s$-plex in $G$. Moreover, $y^*$ selects precisely the missing edges needed to complete the subgraph induced by $\operatorname{supp}(x^*)$ into a clique. This is an existence result and does not imply that every global maximizer has the same combinatorial structure. Indeed, as already pointed out in \cite{cubic}, formulation~\eqref{MS_splex_formulation} may admit spurious global maximizers, namely optimal pairs $(x^*,y^*)$ for which $\operatorname{supp}(x^*)$ is not an $s$-plex in $G$.
% At this point, let us come back to the considered MS formula (\ref{MS_splex_formulation}) for the M$s$PP. As part of the proof of Theorem \ref{MS_splex_theorem}, it is showed that if the $x$-component of the program (\ref{MS_splex_formulation}) is fixed, set for example $x\equiv x^*$, then the following maximization problem is obtained:
% \begin{align}\label{LR_MWBMP_fix_x*}
%     \begin{split}
%         \max \quad & 2\sum_{\{i,j\}\in E} x^*_ix^*_j\ +\ 2\sum_{\{i,j\}\in\bar{E}} x^*_iy_{ij}x^*_j  \\
%         s.t. \quad & B_{\bar{G}} y \leq (s - 1)e,\  y\in [0,1]^{\bar{m}}
%     \end{split}
% \end{align}
% which coincides, up to a constant, with the LR of the MW$b$MP (\ref{max_WbMP}) in the complementary graph $\bar{G}$ with $b=(s-1)e$ and weights $w^*_{ij}=2x^*_ix^*_j,\ \forall\ \{i,j\}\in\bar{E}$. \\
% To conclude the proof, the authors finally showed that it is indeed always possible to find a feasible solution $x^*\in \Delta_n$ such that the above program (\ref{LR_MWBMP_fix_x*}) always admits an integer optimal solution $y^*\in\{0,1\}^{\bar{m}}$.\\
% As already pointed out in \cite{cubic}, a well-known drawback of formulation (\ref{MS_splex_formulation}) is that spurious solutions are admitted, i.e., optimal vectors connected with no $s$-plexes within the graph. 
The following example provides a simple, yet representative illustration of this undesirable behavior. Consider the graph $G=(V,E)$ shown in Figure \ref{spurious_solution_MsPP}. Clearly, the $2$-plex number of $G$ is $w_2^p(G)=3$ since both sets $C=\{1,2,3\}$ and $C'=\{2,3,4\}$ are $2$-plexes of maximum cardinality in $G$. However, there exists a globally optimal solution $(x^*,y^*)\in\Delta_n\times P'_2(G)$ of (\ref{MS_splex_formulation}) whose support in the $x$-component $S^*=\operatorname{supp}(x^*)=\{1,2,3,4\}$ is not a $2$-plex in $G$. Indeed, the solution represented in Figure~\ref{spurious_solution_MsPP} satisfies $f_G(x^*,y^*)=\frac{2}{3}
=1-\frac{1}{w_2^p(G)},$ and is therefore globally optimal by Theorem~\ref{MS_splex_theorem}.
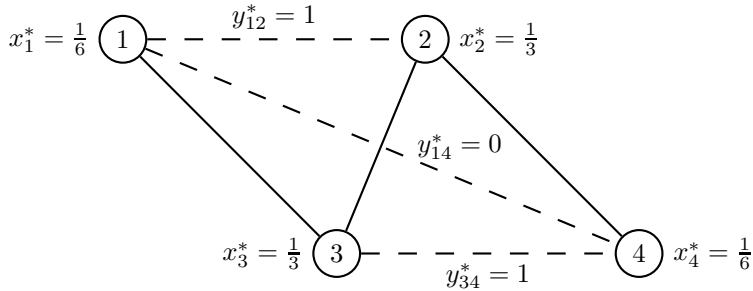
\begin{figure}[ht!]
    \centering
    \begin{tikzpicture}[node distance={40mm}, thick, main/.style = {draw, circle}] 
        \node[main, label=left:{$x^*_1=\frac{1}{6}$}] (1) {$1$};
        \node[main, label=right:{$x^*_2=\frac{1}{3}$}] (2) [right of = 1] {$2$}; 
        \node[main, label=left:{$x^*_3=\frac{1}{3}$}] (3) [below right of = 1] {$3$}; 
        \node[main, label=right:{$x^*_4=\frac{1}{6}$}] (4) [right of = 3] {$4$};
        \draw [dash pattern={on 7pt off 7pt}] (1) -- (2) node [midway, above] {$y^*_{12}=1$}; 
        \draw (1) -- (3); 
        \draw [dash pattern={on 7pt off 7pt}] (1) -- (4) node [midway, right=10pt] {$y^*_{14}=0$};;
        \draw (2) -- (3); 
        \draw (2) -- (4); 
        \draw [dash pattern={on 7pt off 7pt}] (3) -- (4) node [midway, below] {$y^*_{34}=1$};;   
    \end{tikzpicture} 
    \caption{Example of $(x^*,y^*)$ spurious solution for $G$. The dashed lines indicate the missing edges.}
    \label{spurious_solution_MsPP}
\end{figure}
This phenomenon is consistent with the fact that the cubic formulation~\eqref{MS_splex_formulation} for the M$s$PP was derived as a direct extension of the original MS formulation. In particular, it inherits an analogous limitation at the level of solutions: although the formulation provides the exact optimal value, not every global maximizer has a vertex support that gives the desired combinatorial structure.

Accordingly, the primary aim of our study is to develop a regularized variant of formulation~\eqref{MS_splex_formulation} that guarantees the vertex support of each local maximizer forms an $s$-plex in $G$.

\section{General regularized formulation}\label{section_3}
In this section, we address the issue of spurious solutions for the cubic $s$-plex formulation (\ref{MS_splex_formulation}) by introducing suitable regularization terms in the objective function, which ensure equivalence between the derived regularized version and the initial M$s$PP in both a local and global sense. In particular, no post-processing algorithms are needed to recover an $s$-plex structure from the optimal solution returned by the program. 

Let us start by recalling a well-known property of convex maximization problems, which helps to locate optimal solutions as extreme points of the feasible region; see \cite{book} for a more detailed discussion.

\begin{lemma}\label{max_convex_program}
    Let $\Omega\subseteq\mathbb{R}^n$ be a nonempty  polytope, and let $f\in C^2(U)$ be a strictly convex function defined on an open set $U\supseteq\Omega$. Then every local maximizer, and hence every global maximizer, of $\max_{x\in\Omega}f(x)$ is an extreme point of $\Omega$. Moreover, every local maximizer is strict.
\end{lemma}
We continue the discussion by recalling some results from \cite{Hungerford} on the regularized version of the quadratic MS formulation for the MCP, which will indeed play a fundamental role in the development of our subsequent analysis. Given the graph $G=(V,E)$, consider the following optimization problem
\begin{align}\label{regularized_MS_MCP}
    \begin{split}
        \max \quad & \left [ x^\top A_Gx\ +\ \Phi(x) \right ]\\
        s.t. \quad & x\in\Delta_n,
    \end{split}
\end{align}
where $\Phi:\ X\longrightarrow\mathbb{R}$ is a twice continuously differentiable function defined in some open set $X\supseteq\Delta_n$ that satisfies the following conditions for every $x\in\Delta_n$:
\begin{enumerate}[label=(C\arabic*)]
    \item \label{cond_C1} $\nabla^2\Phi(x)>0$, so $\Phi$ is strictly convex over the simplex;
    \item \label{cond_C2} $\|\nabla^2\Phi(x)\|_2<2$, or equivalently $\nabla^2\Phi(x)-2\mathbb{I}_n<0$;
    \item \label{cond_C3} $\Phi$ is constant on the set $\mathcal{P}(x):=\{\tilde{x}\in\Delta_n\ |\ \exists\sigma\in\mathcal{S}_n\ s.t.\ \tilde{x}_i=x_{\sigma(i)}\ \forall i\in\{1,\dots,n\}\}$, where $\mathcal{S}_n$ is the set of permutations of $\{1,\dots,n\}$.
\end{enumerate}
 For simplicity, we now denote by
\begin{itemize}
    \item $k_G(x)=x^\top A_Gx + \Phi(x)$ the objective function of (\ref{regularized_MS_MCP});
    \item $T(x)=\{d\in\mathbb{R}^n\ |\ e^\top d=0\text{, and }d_i\geq0\text{ for all } i\notin\operatorname{supp} (x)  %\text{ whenever }x_i=0
    \}$ the set of feasible directions for (\ref{regularized_MS_MCP}) at $x\in\Delta_n$.
\end{itemize} 
Under these assumptions, the following results hold. Recall that for any $\emptyset\neq C\subseteq V$, the point $x(C)\in \Delta_n$ is the (normalized) characteristic vector of $C$ which satisfies $x(C)_i =\frac{1}{|C|}$ if $i\in C$ and $x(C)_i = 0$ otherwise.
\begin{lemma}[{\cite[Proposition~3]{Hungerford}}] \label{increasing_value_result}
    If $C^1$ and $C^2$ are cliques, then
    \begin{align*}
        |C^1|<|C^2| \Longleftrightarrow k_G(x(C^1))<k_G(x(C^2)).
    \end{align*}
\end{lemma}
\begin{lemma}[{\cite[Lemma~3]{Hungerford}}]\label{negative_scalar_product_with_the_matrix}
    If $C$ is a maximal clique in $G$, then 
    \begin{align*}
        x(C)^\top A_Gd\leq0,\ \forall d\in T(x(C)).
    \end{align*}
\end{lemma}
\begin{lemma}[{\cite[Lemma~4]{Hungerford}}]\label{negative_scalar_product_with_the_penalty_term}
    Let $S\subseteq V$ be non-empty. Then
    \begin{align*}
        \nabla\Phi(x(S))^\top d\leq0,\ \forall d\in T(x(S))\, .
    \end{align*}
    Moreover, the inequality is strict whenever $\operatorname{supp}(d)\nsubseteq S$.
\end{lemma}
Finally, through the use of second-order sufficient optimality conditions, the authors in \cite{Hungerford} were able to show that there exists a one-to-one correspondence between local (global) maximizers of (\ref{regularized_MS_MCP}) and maximal (maximum) cliques in $G$. 

We can now turn to our results. Recall that we have denoted by $A_G$ the adjacency matrix of the graph $G$ and instead by $A_{\bar{G}}(y)$ the induced adjacency matrix of the complementary graph $\bar{G}$ for every $y\in P'_s(G)$. In particular, note that for any binary vector $y^*\in P_s(G)$, we can define the auxiliary graph $G^*$ with adjacency matrix given by $A_{G^*} = A_G + A_{\bar{G}}(y^*)$ as the graph obtained from $G$ by simply adding the missing edges selected by the binary $y^*$ variable. Moreover, let us also indicate with $N(i) = \{j\in V\ |\ \{i,j\}\in E\}$ and $N^*(i) = \{j\in V\ |\ \{i,j\}\in E \cup \operatorname{supp}(y^*)\}$ the neighborhoods of a vertex $i\in V$ for the graphs $G$ and $G^*$, respectively. Finally recall that $\bar{E}(C)$ represents the set of missing edges with both end points in $C$.

Consider the following optimization problem:
\begin{align} \label{general_regularized_formulation_MsPP}
    \begin{split}
        \max\quad& h_G(x,y) \\
        s.t.\quad& (x,y)\in\Delta_n\times P'_s(G),
    \end{split}
\end{align}
where the objective function is given by
\begin{align*}
    h_G(x,y) = f_G(x,y) + \Phi(x) +\Psi(y).
\end{align*}
In particular, the first term corresponds to the objective function (\ref{MS_splex_objective_function}) of the MS formulation for the $s$-plex relaxation, while the remaining terms are general regularizers for the variables $x$ and $y$, separately. The choice of using separate regularization terms is motivated by the different roles and feasible-set geometries of the two variable blocks: $x$ represents the vertices, whereas $y$ represents the missing edges of the graph. Specifically, we consider $\Phi(x)$ as a general regularized function for the MS of the MCP, as introduced in \cite{Hungerford} and reported in the discussion above,  while, for the $y$-component, we introduce the separable regularization function
\begin{align}\label{y_penalty}
    \Psi(y)\ :=\sum_{\{i,j\}\in\bar{E}}\psi(y_{ij}),
\end{align}
where $\psi:\ Y\longrightarrow\mathbb{R}$ is a twice continuously differentiable function, defined on some open set $Y\supseteq[0,1]$, that satisfies the following requirements:
\begin{enumerate}[label=(R\arabic*)]
    \item \label{cond_R1} $\psi(0)=\psi(1)=0$
    \item \label{cond_R2} $\psi'(\frac{1}{2})<-\frac{1}{2}$
    \item \label{cond_R3} $\psi''(t)>0$ for all $t\in[0,1]$, so $\psi$ is strictly convex over the unit interval.
\end{enumerate}
As will be formally shown below, these conditions ensure that the regularization term vanishes at binary vectors while ruling out half-integral extreme points of $P'_s(G)$ as local maximizers. Consequently, the $y$-component of every local maximizer of the regularized formulation is binary.
% As shown below, these conditions ensure that fractional extreme points of $P'_s(G)$ cannot be local maximizers, while $\Psi$ vanishes at every binary vector.

% In this scenario, when the parameter $s=1$, and so the $y$ variables are constrained to be zero, then also $\Psi(y)$ vanishes and the objective function $h_G$ coincides with that one of a regularized version of MS for the MCP. This property in particular shows the robustness of the chosen formulation (\ref{general_regularized_formulation_MsPP}), which can therefore be seen as a direct extension of the regularized version proposed in \cite{Hungerford} and adapted to the $s$-plex case. 

In the special case $s=1$, the $y$-variables are constrained to be zero and, by condition~(R1), $\Psi(y)$ vanishes. Consequently, the objective function $h_G$ reduces exactly to that of the regularized MS formulation for the MCP considered in \cite{Hungerford}. This consistency with the clique case shows that formulation~\eqref{general_regularized_formulation_MsPP} can be viewed as a natural extension of the regularized MS framework from cliques to $s$-plexes.

In the following theorem, we provide a complete characterization of the local
maximizers of formulation~\eqref{general_regularized_formulation_MsPP}. In particular,
we show that every local maximizer is strict, has a binary $y$
component, and has an $x$ component equal to the  characteristic
vector of an $s$-plex.
% In the following, we prove that there exists a one-to-one correspondence between strict local maxima of (\ref{general_regularized_formulation_MsPP}) and characteristic vectors of $s$-plexes within a given graph~$G$. 
In order to show this, we apply second-order sufficient optimality conditions following an approach similar to those in \cite{s-defective, Hungerford}. For this reason, given any feasible point $(x,y)\in\Delta_n\times P'_s(G)$ we define
\begin{align*}
    T(x,y) = \{(\bar{x},\bar{y}) - (x,y)\ |\ (\bar{x},\bar{y})\in\Delta_n\times P'_s(G)\}\subseteq \mathbb{R}^{n+\bar{m}}
\end{align*}
as the set of feasible directions at $(x,y)$, while for any vector $r\in \mathbb{R}^{n+\bar{m}}$ we have
\begin{align*}
    T^0((x,y),r) = \{d\in T(x,y)\ |\ d^\top r=0 \}
\end{align*}
as the intersection between the set $T(x,y)$ and the hyperplane orthogonal to $r$. So, in this setting, a feasible point $(x^*,y^*)\in\Delta_n\times P'_s(G)$ is a strict local maximizer of (\ref{general_regularized_formulation_MsPP}) if the following conditions hold:
\begin{align}\label{general_second_order_sufficient_condition_1}
    \nabla h_G(x^*,y^*)^\top  d\leq 0,\ \forall d\in T(x^*,y^*),
\end{align}
\begin{align}\label{general_second_order_sufficient_condition_2}
    d^\top  \nabla^2h_G(x^*,y^*)d<0,\ \forall d \in T^0((x^*,y^*),\nabla h_G(x^*,y^*))\setminus\{0\}\, .
\end{align}
%The following  provides a complete characterization of local maximizers of the regularized formulation (\ref{general_regularized_formulation_MsPP}).

 \begin{theorem} \label{general_local_equiv_prop}
    The following statements are equivalent:
    \begin{itemize}
        \item[i.] $(x^*,y^*)\in \Delta_n\times P'_s(G)$ is a local maximizer of (\ref{general_regularized_formulation_MsPP});
        \item[ii.] $(x^*,y^*)\in \Delta_n\times P'_s(G)$ is a strict local maximizer of (\ref{general_regularized_formulation_MsPP});
        \item[iii.] $y^*$ is binary and $x^*=x(C)$ for $C$ an $s$-plex in $G$ which corresponds to a maximal clique in the auxiliary graph $G^*$ induced by $y^*$. Moreover $y^*_{ij} = 1$ for every $\{i,j\} \in \bar{E}(C)$.
    \end{itemize}
    In either of these equivalent cases we have that
   $h_G(x^*,y^*) = 1 - \frac{1}{|C|} + \Phi(x(C)).
   $
\end{theorem}
\begin{proof}

    \fbox{$(i)\Longrightarrow(iii)$}  Let $(x^*,y^*)\in \Delta_n\times P'_s(G)$ be a local maximizer of (\ref{general_regularized_formulation_MsPP}). First, we want to show that $y^*\in P_s(G)$. In order to do this, let us fix the $x$-component of the problem, so $x\equiv x^*$, thus deriving the following optimization problem in the single $y$ variable:
    \begin{align}\label{general_regularized_formulation_x_fixed}
        \begin{split}
            &\max\quad g_G(y)  \\
            &s.t.\quad\ \ y\in P'_s(G),
        \end{split}
    \end{align}
    where the objective function is given by
    \begin{align*}
        g_G(y) = h_G(x^*,y) = {(x^*)}^\top [A_G+A_{\bar{G}}(y)]x^* + \Phi(x^*)\ + \sum_{\{i,j\}\in\bar{E}}\psi(y_{ij}).
    \end{align*}
    Since $(x^*,y^*)$ is a local maximizer of (\ref{general_regularized_formulation_MsPP}) by assumption, $y^*$ must also be a local maximizer of (\ref{general_regularized_formulation_x_fixed}). Moreover, the objective $g_G(y)$ has a positive-definite Hessian matrix over the polytope $P'_s(G)$, since for any $\{i,j\}\in\bar{E}$ and $y\in P'_s(G)$ we have
    $
        \frac{\partial^2}{\partial y_{ij}^2}g_G(y) = \psi''(y_{ij})>0,
   $
    while for any pair of distinct missing edges $\{i,j\},\{k,l\}\in\bar{E}$ the corresponding entry is
    $
        \frac{\partial^2}{\partial y_{ij}\partial y_{kl}}g_G(y) = 0.
    $
    Hence, by Lemma \ref{max_convex_program} we can conclude that all local maximizers of (\ref{general_regularized_formulation_x_fixed}) are extreme points of the feasible region. This in particular implies that the local maximizer $y^*$  is a vertex of the polytope $P'_s(G)$. 
    
    On the other hand, we have already pointed out how the feasible region $P'_s(G)$ coincides with a fractional $b$-matching polytope in the complementary graph $\bar{G}$ with $b=(s-1)e$, see problem~(\ref{LR_MWBMP_fix_x*}) above. Thus, by Lemma~\ref{lemma_fractional_b-matching_polytope} we know that the vertices of $P'_s(G)$ are half-integral, i.e. $y^*_{ij}\in\big\{0,\frac{1}{2},1\big\}$ for every $\{i,j\}\in\bar{E}$. Let us assume, by contradiction, that there exists an edge $\{i,j\}\in\bar{E}$ such that $y^*_{ij}=\frac{1}{2}$ is fractional. By
Lemma~\ref{lemma_fractional_b-matching_polytope}, the fractional support of $y^*$ contains an odd cycle,
and every edge of this cycle has value $1/2$. Fix one such edge
$\{i,j\}\in\bar E$ and let
\[
\hat y:=y^*-\varepsilon\, e_{ij},
\qquad 0<\varepsilon\leq\frac12,
\]
where $e_{ij}\in\mathbb{R}^{\bar m}$ denotes the unit vector associated with the missing edge $\{i,j\}$. Since the constraints defining $P'_s(G)$ are upper-bound
constraints with non-negative coefficients, $\hat y\in P'_s(G)$. By the convexity of $\psi$,
\[
\psi\left(\frac12-\varepsilon\right)
-\psi\left(\frac12\right)
\geq
-\varepsilon\psi'\left(\frac12\right).
\]
Consequently,
\begin{align*}
g_G(\hat y)-g_G(y^*)
&=
-2\varepsilon x_i^*x_j^*
+\psi\left(\frac12-\varepsilon\right)
-\psi\left(\frac12\right)\\
&\geq
-\varepsilon\left(
2x_i^*x_j^*+\psi'\left(\frac12\right)
\right)>0,
\end{align*}
where the last inequality follows from
$2x_i^*x_j^*\leq 1/2$ and condition~(R2).
This contradicts the local optimality of $y^*$ and finally shows that  $y^*_{ij}\in\{0,1\}$ for every $\{i,j\} \in \bar{E}$, so $y^*$ is binary. \\
After showing this first property, let us fix the $y$-component of (\ref{general_regularized_formulation_MsPP}) instead, so $y\equiv y^*$. Then $x^*$ must be a local maximizer of the function 
    \begin{align*}
        h_G(x,y^*)= x^\top (A_G + A_{\bar{G}}(y^*))x + \Phi(x) = x^\top A_{G^*}x + \Phi(x)
    \end{align*}
    over the standard simplex, where $G^*$ is the auxiliary graph induced by $y^*\in P_s(G)$. As we can see, the considered optimization problem coincides with the regularized version of the maximum clique formulation (\ref{regularized_MS_MCP}) over $G^*$ given in \cite{Hungerford}, thus it must hold that $x^*=x(C)$ for $C$ a maximal clique in $G^*$. Since the auxiliary graph $G^*$ is defined by adding the missing edges in $G$ corresponding to the $\operatorname{supp}(y^*)$, then $C$ must also be an $s$-plex in $G$ and $y^*_{ij} = 1$ for all $\{i,j\}\in\bar{E}(C)$. \\

    \fbox{$(iii)\Longrightarrow(ii)$} Let $(x^*,y^*)\in\Delta_n\times P'_s(G)$ satisfy the properties outlined in the  statement $(iii)$ above. To show that it is a strict local maximizer of $(\ref{general_regularized_formulation_MsPP})$, we want to verify that the second-order sufficient optimality conditions (\ref{general_second_order_sufficient_condition_1}) and (\ref{general_second_order_sufficient_condition_2}) hold in this case. 
    
    For any missing edge $\{i,j\}\in\bar{E}$ we can compute 
    \begin{align*}
        \frac{\partial}{\partial y_{ij}}h_G(x^*,y^*) = 2x^*_ix^*_j + \psi'(y^*_{ij}).
    \end{align*}
    In the following we denote by $S:= \operatorname{supp}(y^*)$, and $\bar{S}:=\bar{E}\setminus S$. Since by assumption $y_{ij}^*=1$ for any $\{i,j\}\in\bar{E}(C)$, clearly $2x_l^*x_k^*=0$ for any $\{l,k\}\in\bar{S}$. Moreover, by conditions \ref{cond_R1} and \ref{cond_R3} we necessarily have $\psi'(0)<0$ and $\psi'(1)>0$. Thus, for $\{l,k\}\in\bar{S}$ we get
    \begin{align}\label{general_derivate_by_y_not_in_S}
        \frac{\partial}{\partial y_{lk}}h_G(x^*,y^*) = \psi'(0)<0,
    \end{align}
    while for $\{i,j\}\in S$
    \begin{align}\label{general_derivate_by_y_in_S}
        \frac{\partial}{\partial y_{ij}}h_G(x^*,y^*) = 2x^*_ix^*_j + \psi'(1) >0,
    \end{align}
    where we use the fact that $y^*\in P_s(G)$ is binary.  
    
    Now, let $d$ be a feasible direction at $(x^*,y^*)$, such that $d=(\bar{x},\bar{y}) - (x^*,y^*)$ for some $(\bar{x},\bar{y})\in\Delta_n\times P'_s(G)$. We can separately consider $d=(d_x,d_y)=(\bar{x}-x^*,\bar{y}-y^*)$, where $d_x$ and $d_y$ are feasible directions for the $x$ and $y$ components of the optimization problem at $x^*$ and $y^*$, respectively. 
Furthermore, for any $U\subseteq V$, we denote by $d_U:=(d_x)_U$ the restriction of the $x$-component of $d$ to the coordinates indexed by $U$.

    Then we can compute
    \begin{align}\label{first_order_condition_in_x}
        \begin{split}
            \nabla_xh_G(x^*,y^*)^\top d_x &= (2A_Gx^*+2A_{\bar{G}}(y^*)x^*+\nabla\Phi(x^*))^\top d_x \\ &= 2x(C)^\top A_{G^*}d_x + \nabla\Phi(x(C))^\top d_x\leq0,
        \end{split}
    \end{align}
    where the inequality follows by Lemma \ref{negative_scalar_product_with_the_matrix} and Lemma \ref{negative_scalar_product_with_the_penalty_term} with respect to the MCP over the auxiliary graph $G^*$ induced by $y^*$. 
    
    For the $y$-component we can instead compute
    \begin{align*}
            (\bar{y} - y^*)^\top \nabla_yh_G(x^*,y^*) &= \sum_{\{i,j\}\in \bar{E}} (\bar{y}_{ij}-y^*_{ij}) \frac{\partial}{\partial y_{ij}}h_G(x^*,y^*) \\
            &= \sum_{\{i,j\}\in S} (\bar{y}_{ij}-1) \frac{\partial}{\partial y_{ij}}h_G(x^*,y^*)\ + \sum_{\{i,j\}\in\bar{S}} \bar{y}_{ij} \frac{\partial}{\partial y_{ij}}h_G(x^*,y^*)
    \end{align*}
    thus
    \begin{align}\label{first_order_condition_in_y}
        \begin{split}
            \nabla_yh_G(x^*,y^*)^\top d_y \leq 0
        \end{split}
    \end{align}
    by inequalities (\ref{general_derivate_by_y_not_in_S}), (\ref{general_derivate_by_y_in_S}), and the fact that $\bar{y}_{ij}\leq 1$. 
    
To summarize, combining~(\ref{first_order_condition_in_x}) and (\ref{first_order_condition_in_y}) we can conclude that
    \begin{align*}
        \nabla h_G(x^*,y^*)^\top d = \nabla_xh_G(x^*,y^*)^\top d_x + \nabla_yh_G(x^*,y^*)^\top d_y \leq 0.
    \end{align*}
    This finally proves that first-order condition (\ref{general_second_order_sufficient_condition_1}) is satisfied.

Now, in order to show that the second condition
(\ref{general_second_order_sufficient_condition_2}) is also satisfied,
notice that both terms in
\[
\nabla h_G(x^*,y^*)^\top d
=
\nabla_xh_G(x^*,y^*)^\top d_x
+
\nabla_yh_G(x^*,y^*)^\top d_y
\]
are nonpositive by (\ref{first_order_condition_in_x}) and
(\ref{first_order_condition_in_y}). Therefore, equality can hold only if
both terms vanish. By the strict part of Lemma
\ref{negative_scalar_product_with_the_penalty_term} and the strict
inequalities
(\ref{general_derivate_by_y_not_in_S}) and
(\ref{general_derivate_by_y_in_S}), this requires
$(\bar{x},\bar{y})$ to satisfy
\begin{itemize}
    \item $\operatorname{supp}(\bar{x})\subseteq C$;
    \item $\bar{y}_{ij}=1$ for every $\{i,j\}\in S$, while
    $\bar{y}_{lk}=0$ for every $\{l,k\}\in\bar{S}$.
\end{itemize}
In particular, the second condition is equivalent to $\bar{y}=y^*$.
    % Now, in order to show that the second condition (\ref{general_second_order_sufficient_condition_2}) is also valid, notice that in the final scalar product we can have equality if and only if there are equalities in both (\ref{first_order_condition_in_x}) and (\ref{first_order_condition_in_y}), so if and only if $(\bar{x},\bar{y})$ is a feasible solution such that:
    % \begin{itemize}
    %     \item $\operatorname{supp}(\bar{x}) \subseteq C$, by the second statement in Lemma \ref{negative_scalar_product_with_the_penalty_term}
    %     \item $\bar{y}_{ij} = 1$ for $\{i,j\}\in S$, while $\bar{y}_{lk}=0$ for $\{l,k\}\in\bar{S}$.
    % \end{itemize}
    % In particular, this second condition implies that $\bar{y}=y^*$. 
We now show that these conditions are also sufficient for equality. Indeed, if $\operatorname{supp}(\bar{x})\subseteq C$, then $d_x=\bar{x}-x^*$ is supported on $C$ and \[ e^\top d_C = e^\top\bar{x}_C-e^\top x_C^* = 1-1=0. \] Since $C$ is a clique in the auxiliary graph $G^*$, the components of the vector $A_{G^*}x(C)$ indexed by $C$ are all equal to $\frac{|C|-1}{|C|}$. Moreover, condition \ref{cond_C3} implies that the components of $\nabla\Phi(x(C))$ indexed by $C$ are equal, since $x(C)$ is unchanged by every permutation of the vertices within $C$. Therefore, there exists $\lambda\in\mathbb{R}$ such that $ \nabla_C\Phi(x(C))=\lambda e, $ and hence \begin{align*} \nabla_xh_G(x^*,y^*)^\top d_x &= 2x(C)^\top A_{G^*}d_x +\nabla\Phi(x(C))^\top d_x\\ &= \left( 2\frac{|C|-1}{|C|}+\lambda \right)e^\top d_C =0. \end{align*} Furthermore, $\bar{y}=y^*$ implies $d_y=\underline{0}$ and therefore $ \nabla_yh_G(x^*,y^*)^\top d_y=0. $

Thus, the two conditions above are also sufficient for equality, and we obtain \begin{align*} T^0((x^*,y^*),\nabla h_G(x^*,y^*)) &= \left\{ d\in T(x^*,y^*) \ \middle|\ d=(\bar{x},\bar{y})-(x^*,y^*),\ \bar{x}_{\bar C}=\underline{0},\ \bar{y}=y^* \right\}\\ &= \left\{ d=(d_x,d_y)\in T(x^*,y^*) \ \middle|\ d_{\bar C}=\underline{0},\ d_y=\underline{0} \right\}. \end{align*}
     % Thus, $(x^*,y^*)$ is a first-order stationary point with:
     % \begin{align*}
     %     T^0((x^*,y^*),\nabla h_G(x^*,y^*)) & = \{d\in T(x^*,y^*)\ |\ d=(\bar{x},\bar{y})-(x^*,y^*),\ \bar{x}_{\bar{C}} =\underline{0},\ \bar{y} = y^*\} 
     %     \\ 
     %    &=  \{d\in T(x^*,y^*)\ |\ d_{\bar{C}}=d_S=d_{\bar{S}}=\underline{0}\}.
     % \end{align*}
For any $U\subseteq V$, we denote by $\nabla^2_{UU}h_G(x^*,y^*)$ the principal submatrix of $\nabla^2h_G(x^*,y^*)$ corresponding to the $x$-variables indexed by $U$. In particular, 
 % For the sake of simplicity, let $\nabla^2_{CC}h_G$ denote the submatrix of $\nabla^2h_G$ corresponding to the $x$-variables indexed by $C$. F
    for any distinct pair of vertices $i,j\in C$, we have \begin{align*} \nabla^2_{x_ix_j}h_G(x^*,y^*) = 2+\nabla^2_{x_ix_j}\Phi(x^*) \end{align*} since $C$ is a clique in the auxiliary graph $G^*$, while \begin{align*} \nabla^2_{x_ix_i}h_G(x^*,y^*) = \nabla^2_{x_ix_i}\Phi(x^*) \end{align*} for every $i\in C$. Therefore, \begin{align*} \nabla^2_{CC}h_G(x^*,y^*) = 2ee^\top-2\mathbb{I}_{|C|} +\nabla^2_{CC}\Phi(x^*). \end{align*}
    % For the sake of simplicity, let $\nabla^2_{CC}h_G$ denoting the sub-matrix of the Hessian matrix $\nabla^2h_G$ with both indexes in $C$. For any distinct pair of vertices $i,j\in C$ we have that
    % \begin{align*}
    %     \nabla^2_{x_ix_j}h_G(x^*,y^*)=2+\nabla^2_{x_ix_j}\Phi(x^*)
    % \end{align*}
    % since $C$ is a clique in the auxiliary graph $G^*$, while
    % \begin{align*}
    %     \nabla^2_{x_ix_i}h_G(x^*,y^*)=\nabla^2_{x_ix_i}\Phi(x^*)
    % \end{align*}
    % for any fixed vertex $i\in C$. So, we can rewrite the considered sub-matrix as
    % \begin{align*}
    %     \nabla^2_{CC}h_G(x^*,y^*)=2ee^\top  - 2\mathbb{I}_{|C|} + \nabla^2\Phi(x^*).
    % \end{align*}
Finally, let \[ d\in T^0((x^*,y^*),\nabla h_G(x^*,y^*))\setminus\{0\}. \] By the preceding characterization, $d_y=\underline{0}$ and $d_{\bar C}=\underline{0}$. Hence, only the components of $d_x$ indexed by $C$ contribute to the quadratic form, and we have $ e^\top d_C=0 $ because $d_x$ is a feasible direction for the standard simplex. Consequently, \begin{align*} d^\top\nabla^2h_G(x^*,y^*)d &= d_C^\top\nabla^2_{CC}h_G(x^*,y^*)d_C\\ &= d_C^\top \left( 2ee^\top-2\mathbb{I}_{|C|} +\nabla^2_{CC}\Phi(x^*) \right)d_C\\ &= d_C^\top \left( \nabla^2_{CC}\Phi(x^*) -2\mathbb{I}_{|C|} \right)d_C<0. \end{align*} The strict inequality follows from condition \ref{cond_C2}, since every principal submatrix of a negative definite matrix is negative definite. This proves that the second condition (\ref{general_second_order_sufficient_condition_2}) is satisfied.

    % Finally, taking any direction $d\in T^0((x^*,y^*),\nabla h_G(x^*,y^*))$ we can compute
    % \begin{align*}
    %     d^\top \nabla^2h_G(x^*,y^*)d & = d_C^\top \nabla^2_{CC}h_G(x^*,y^*)d_C  \\
    %     & = d_C^\top (2ee^\top  - 2\mathbb{I}_{|C|} + \nabla_{CC}^2\Phi(x^*))d_C \\
    %     &= d_C^\top (\nabla_{CC}^2\Phi(x^*) - 2\mathbb{I}_{|C|})d_C < 0
    % \end{align*}
    % since $\Phi$ satisfies condition \ref{cond_C2} and principal sub-matrices of a negative definite matrix are still negative definite. Moreover, the last equality holds because
    % \begin{align*}
    %     e^\top d_C = e^\top (\bar{x}_C - x^*_C) = e^\top \bar{x}_C - e^\top x^*_C = 1 - 1 = 0
    % \end{align*}
    % since $\operatorname{supp}(\bar{x}),\operatorname{supp}(x^*)\subseteq C$. This finally shows that also the second condition (\ref{general_second_order_sufficient_condition_2}) is satisfied.   
    To summarize, we have proved that the given point $(x^*,y^*)$ satisfies the second-order sufficient optimality conditions (\ref{general_second_order_sufficient_condition_1}) and (\ref{general_second_order_sufficient_condition_2}) for the considered regularized formulation (\ref{general_regularized_formulation_MsPP}) and hence it is a strict local maximizer.\\
    
    \fbox{$(ii)\Longrightarrow(i)$} is trivial.
\end{proof}

Theorem~\ref{general_local_equiv_prop} highlights a key advantage of the regularized version (\ref{general_regularized_formulation_MsPP}), namely that, unlike the original formulation (\ref{MS_splex_formulation}), it does not admit spurious local solutions, since every local maximizer corresponds to an $s$-plex in the given graph. Moreover, the $x$-component of the solution coincides with the characteristic vector of the detected $s$-plex structure. 

We conclude this section with an extension of the previous result to the global equivalence case.
\begin{corollary}\label{corollary_global_equivalence}
    A feasible point $(x^*,y^*)\in\Delta_n\times P'_s(G)$ is a global maximizer of (\ref{general_regularized_formulation_MsPP}) if and only if $y^*$ is binary and $x^*=x(C)$ for a maximum $s$-plex $C$ in $G$; moreover $y^*_{ij}=1$ for every $\{i,j\}\in\bar{E}(C)$.
\end{corollary}
\begin{proof}
    Let $(x^*,y^*)$ be a global maximizer of (\ref{general_regularized_formulation_MsPP}). By Theorem~\ref{general_local_equiv_prop} we have that $y^*$ is binary, $x^*=x(C)$ for an $s$-plex $C$  in $G$ and $y^*_{ij}=1$ for every $\{i,j\}\in\bar{E}(C)$. Moreover, the corresponding value of the objective function is 
    \begin{align*}
        h_G(x^*,y^*) = 1 - \frac{1}{|C|} + \Phi(x(C)).
    \end{align*}
    Assume by contradiction that there exists an $s$-plex $C^1\subseteq V$  in $G$ such that $|C^1|>|C|$. By the heredity property of the $s$-plex model (see \cite{clique_relaxations} and \cite{heredity} for further details) there exists an $s$-plex $C^2\subset C^1$ such that $|C|=|C^2|$. Therefore, let $(x^1,y^1)$ be the solution corresponding to the $s$-plex $C^1$, where $x^1=x(C^1)$ and $y^1_{ij}=1$ if $\{i,j\}\in\bar{E}(C^1)$ and $0$ otherwise, while $(x^2,y^2)=(x(C^2),y^1)$ is a feasible solution of (\ref{general_regularized_formulation_MsPP}) corresponding to $C^2$. With these choices, we can compute
    \begin{align*}
        h_G(x^*,y^*) = 1-\frac{1}{|C|}+\Phi(x(C)) = 1-\frac{1}{|C^2|}+\Phi(x(C^2))=h_G(x^2,y^2)
    \end{align*}
    since $\Phi$ satisfies \ref{cond_C3} and $x(C^2)\in\mathcal{P}(x(C))$. Moreover, if we now fix the $y$ component of problem~(\ref{general_regularized_formulation_MsPP}) as $y\equiv y^1$, then $(x^1,y^1)$ and $(x^2,y^2)$ are both feasible solutions for the regularized version of the MS for the MCP over the auxiliary graph $G^1$ induced by $y^1$, so
    \begin{align*}
    \begin{split}
        \max \quad & k_{G^1}(x) \\
        s.t. \quad & x\in\Delta_n
    \end{split}
    \end{align*}
    where the objective function is given by $k_{G^1}(x):=x^\top A_{G^1}x\ +\ \Phi(x)$. Finally, by Lemma \ref{increasing_value_result}, since $|C^2|<|C^1|$ by construction, we get
    \begin{align*}
        k_{G^1}(x^2)<k_{G^1}(x^1) &\Longleftrightarrow h_G(x^2,y^1)<h_G(x^1,y^1) \\
        &\Longleftrightarrow h_G(x^2,y^2)<h_G(x^1,y^1) \\ 
        &\Longleftrightarrow h_G(x^*,y^*)<h_G(x^1,y^1)
    \end{align*}
    leading to a contradiction with the global optimality of $(x^*,y^*)$.\\
    The converse follows immediately from the fact that any vector $(x^*,y^*)\in\Delta_n\times P'_s(G)$ satisfying the properties outlined in the above statement has the same value of the objective function $h_G$ as any global maximizer of (\ref{general_regularized_formulation_MsPP}).
\end{proof}

\section{Regularization functions}\label{section_4}
We present three classes of regularization terms satisfying conditions \ref{cond_C1}-\ref{cond_C3}:
\begin{itemize}
    \item $\Phi_1(x):=\frac{\alpha}{2}\|x\|_2^2$, with $\alpha\in(0,2)$;
    \item $\Phi_2(x):=\alpha_1\|x+\varepsilon e\|_p^p$ with $\varepsilon>0$, $p>2$ and $0<\alpha_1<\frac{2}{p(p-1)(1+\varepsilon)^{p-2}}$;
    \item $\Phi_3(x):=\alpha_2\sum_{i=1}^n(\exp(-\beta x_i)-1)$ with $\beta>0$ and $0<\alpha_2<\frac{2}{\beta^2}$.
\end{itemize}
Here, $\Phi_1$ for $\alpha=1$ is the two-norm regularization function introduced by Bomze \cite{evolution}. The validity of these classes, with respect to conditions \ref{cond_C1}-\ref{cond_C3}, was established in \cite{Hungerford}, to which we refer the reader for the corresponding proofs.

On the other hand, for the $y$-regularization term, we can define the following three families of regularization functions $\psi_\ell:Y_\ell\to\mathbb{R}$, where $Y_\ell$ is an open
interval containing $[0,1]$,
% $\psi(t):[0,1]\subseteq Y\longrightarrow\mathbb{R}$ 
with the indicated choices of parameters:
\begin{enumerate}
    % \item $\psi_1(t):=b_{\IB{\nu}}t(t-1)(t^{n-2} + a_{\IB{\nu}})$, for any integer $n>2$ with $a_{\IB{\nu}}>\frac{n-1}{2}\left(\frac{n-3}{n}\right)^{n-3}$ and $b_{\IB{\nu}}>\frac{2^{n-2}}{n-2}$
    \item $\psi_1(t):=b_{\nu}t(t-1)(t^{\nu-2}+a_{\nu})$, where:
\begin{itemize}
    \item for $\nu=3$, $a_3>1$ and $b_3>2$;
    \item for any integer $\nu\geq 4$ with $a_{\nu}>\frac{\nu-1}{2}\left(\frac{\nu-3}{\nu}\right)^{\nu-3}$ and $b_{\nu}>\frac{2^{\nu-2}}{\nu-2}$;
\end{itemize}
    \item $\psi_2(t):=B_k\left(\exp(kt)-1-(\exp(k)-1)t\right)$,
    % $\psi_2(t):=B_k(e^{kt}-1-(e^k-1))t$
     with $k>0$ and $B_k>\frac{1}{2}(\exp(k)-k\exp(k/2)-1)^{-1}$;
    \item $\psi_3(t):=D_q(\log(1-q)t - \log(1-tq))$, with $0<q<1$ and \\$D_q> -\frac{1}{2}\left(\log(1-q) + \frac{2q}{2-q}\right)^{-1}$.
\end{enumerate}
In the following, we show that each class satisfies conditions \ref{cond_R1}-\ref{cond_R3}. 

First, note that all three terms satisfy \ref{cond_R1} trivially. We now address the remaining conditions separately for function $\psi_1$. After some manipulations, we get 
\begin{align*}
    \psi_1'(t) = b_{\nu}(\nu t^{\nu-1} - (\nu-1)t^{\nu-2} + 2a_{\nu}t - a_{\nu}),
\end{align*}
where substituting $t=\frac{1}{2}$ leads to
\begin{align*}
    \psi_1'\left(\frac{1}{2}\right) = b_{\nu}(2-\nu)\left(\frac{1}{2}\right)^{\nu-1}<0,
\end{align*}
if we choose $b_{\nu}>0$, for any integer $\nu>2$ and independently of the parameter $a_{\nu}$. Thus, condition \ref{cond_R2} is satisfied if
\begin{align*}
    b_{\nu}(2-\nu)\left(\frac{1}{2}\right)^{\nu-1}<\ -\frac{1}{2}\  \Longleftrightarrow\ b_{\nu}>\frac{2^{\nu-2}}{\nu-2}.
\end{align*}
We now verify condition~\ref{cond_R3}. The case $\nu=3$ needs to be treated separately. Here,
\[
\psi_1''(t)=2b_3(3t+a_3-1),
\]
whose minimum over $[0,1]$ is attained at $t=0$. Therefore,
$\psi_1''(t)>0$ for every $t\in[0,1]$ whenever $a_3>1$.

Suppose now that $\nu\geq4$. We have
\[
\psi_1''(t)
=
b_{\nu}\left((\nu-1)t^{\nu-3}(\nu t-\nu+2)+2a_{\nu}\right).
\]
Now, verifying the correctness of condition \ref{cond_R3} is rather non-trivial since we also want to determine the best lower bound for the parameter $a_{\nu}$ in order to ensure the strict convexity of $\psi_1$ over the unit interval.
To ensure that $\psi_1''(t)>0$ for all $t\in[0,1]$, we first determine the minimum value of $\psi_1''$ on the unit interval. For this purpose, we separately consider the term of $\psi_1''(t)$ that depends explicitly on $t$:
\begin{align*}
    g(t):=t^{\nu-3}(\nu t-\nu+2).
\end{align*}
Since we want to determine the minimum value of the function $g$, good candidate points are on the boundary of the unit interval, where we have $g(0)=0$ and $g(1)=2$, or on the set of stationary points. So, let us compute the derivative
\begin{align*}
    g'(t) = (\nu-3)t^{\nu-4}(\nu t-\nu+2) + \nu t^{\nu-3}.
\end{align*}
One can easily check that 
\begin{align*}
    g'(t^*)=0\Longleftrightarrow t^*=\frac{\nu-3}{\nu}
\end{align*}
and in this case the function value is given by
\begin{align*}
    g(t^*)=-\left(\frac{\nu-3}{\nu}\right)^{\nu-3}<0
\end{align*}
showing that $t^*$ is the minimizer in the interval $[0,1]$ for the function $g(t)$. Thus, we finally get
\begin{align*}
    \psi_1''(t)>0&\Longleftrightarrow\psi_1''(t^*)>0\\
    &\Longleftrightarrow-(\nu-1)\left(\frac{\nu-3}{\nu}\right)^{\nu-3} + 2a_{\nu}>0 \\
    &\Longleftrightarrow a_{\nu}>\frac{\nu-1}{2}\left(\frac{\nu-3}{\nu}\right)^{\nu-3}
\end{align*}
proving the correctness of condition \ref{cond_R3} with the chosen parameter $a_{\nu}$. 
Although the lower-bound sequence
\[
L_{\nu}:=\frac{\nu-1}{2}
\left(\frac{\nu-3}{\nu}\right)^{\nu-3}
\]
is not monotone for all $\nu$, it diverges to $+\infty$ as $\nu\to+\infty$. Consequently, any sequence of admissible parameters satisfying $a_{\nu}>L_{\nu}$ also diverges.

% Notice that the obtained sequence $\{a_{\IB{\nu}}\}_{n\geq3}$ is strictly increasing since $a_{n+1}>a_{\IB{\nu}}$ for all $n\geq3$. Moreover, it is a divergence sequence as $\lim_{n\to+\infty}a_{\IB{\nu}}=+\infty$. 

In the remainder, let us instead verify the correctness of conditions \ref{cond_R2} and \ref{cond_R3} for functions $\psi_2$ and $\psi_3$ simultaneously. First, note that the function $\psi_3$ is well defined only for $q<1$. Then, by computing the derivatives in the middle of the unit interval, we have
\begin{align*}
    &\psi_2'\left(\frac{1}{2}\right)=B_k(k\exp(k/2)-\exp(k)+1) \\
    &\psi_3'\left(\frac{1}{2}\right)=D_q\left(\log(1-q)+\frac{2q}{2-q}\right)
\end{align*}
Since $k>0$ and $0<q<1$, both quantities in parentheses are negative.
Therefore, condition~\ref{cond_R2} is satisfied provided that
% so, if we choose $B_k,\IB{D_q}>0$, then by imposing condition \ref{cond_R2} we get $k>0$ and $0<q<1$ for $\psi_2$ and $\psi_3$, respectively. Thus, we can derive the conditions on the parameters as
\begin{align*}
    &B_k>\frac{1}{2(\exp(k)-k\exp(k/2)-1)} \\
    &D_q> -\frac{1}{2}\left(\log(1-q) + \frac{2q}{2-q}\right)^{-1} .
\end{align*}
Finally, condition \ref{cond_R3} is always satisfied for all $t\in[0,1]$ with the chosen parameters, because the second derivatives satisfy
\begin{align*}
    &\psi_2''(t)=B_kk^2\exp(kt) >0 \\
    &\psi_3''(t)= D_q\frac{q^2}{(1-qt)^2}>0 .
\end{align*}
Note that if the parameters $B_k$ and $D_q$ are non positive, then condition \ref{cond_R3} can never be satisfied.

\section{FW variant for the M\texorpdfstring{$s$}{s}PP}\label{section_5}
In the analysis that follows, we focus on the objective function below, which we adopt as a representative example of the class being studied:
\begin{align}
    \label{particular_objective}
    h_G(x,y) = x^\top (A_G + A_{\bar{G}}(y))x + \frac{\alpha}{2}\|x\|_2^2 + \beta\sum_{\{i,j\}\in\bar{E}}y_{ij}(y_{ij}-1)(y_{ij} + \gamma)
\end{align}
with $\alpha\in(0,2)$, $\beta>2$ and $\gamma>1$. This objective is obtained by combining $\Phi_1$ with the cubic
regularization $\psi_1$ corresponding to $\nu=3$, $a_3=\gamma$ and $b_3=\beta$. Accordingly, the conditions
$\alpha\in(0,2)$, $\beta>2$, and $\gamma>1$ ensure that
\ref{cond_C1}-\ref{cond_C3} and
\ref{cond_R1}-\ref{cond_R3} are satisfied. 
We note that the analysis below can also be extended to the other families of regularization functions introduced in Section~\ref{section_4}, although the corresponding bounds depend on the chosen family and must be derived separately.

We first recall the notion of stationarity used throughout this section.

\begin{definition}[First-order stationary point]\label{def_stationary_point}
Let $\Omega\subseteq\mathbb{R}^p$ be a nonempty compact convex set, and
let $F$ be continuously differentiable on an open set containing
$\Omega$. A point $z\in\Omega$ is called stationary for the
maximization of $F$ over $\Omega$ if
\begin{align*}
    \nabla F(z)^\top(u-z)\leq0
    \qquad\text{for every }u\in\Omega.
\end{align*}
\end{definition}

 We now describe the away-step Frank-Wolfe (AFW) method over the standard simplex, as introduced in \cite{AFW} and summarized in Algorithm \ref{AFW} below, since it plays a key role in the development of a tailored FW algorithm to solve the M$s$PP. 

In a general setting, we consider the optimization problem
\begin{align*}
    \begin{split}
        \max \quad & f(w) \\
        s.t. \quad & w\in\Delta_n
    \end{split}
\end{align*}
where we assume $f$ to be a continuously differentiable function on the standard simplex $\Delta_n$.
\begin{algorithm}
\caption{AFW on the standard simplex}
\label{AFW}
\begin{algorithmic}[1]
\STATE Initialize $w^{(0)}\in\Delta_n$, $k:=0$.
\STATE Set $s^{(k)}:=e_{\hat{i}}$, with
$\hat{i}\in\arg\max_{i=1,\dots,n}\nabla_i f(w^{(k)})$, and
$d_{FW}^{(k)}:=s^{(k)}-w^{(k)}$. \label{step:FWdir}
\STATE \textbf{If} $w^{(k)}$ is stationary, \textbf{STOP}.
\STATE Set $v^{(k)}:=e_{\hat{j}}$, with
$\hat{j}\in\arg\min_{j\in S_k}\nabla_j f(w^{(k)})$,
where $S_k:=\{j:w_j^{(k)}>0\}$, and
$d_A^{(k)}:=w^{(k)}-v^{(k)}$. \label{step:AFWdir}
% \STATE Set $(d^{(k)},\alpha_k^{\max}):=(d_{FW}^{(k)},1)$ if
% $\nabla f(w^{(k)})^\top d_{FW}^{(k)}
% \geq\nabla f(w^{(k)})^\top d_A^{(k)}$; otherwise set
% $(d^{(k)},\alpha_k^{\max})
% :=\left(d_A^{(k)},\dfrac{w_{\hat{j}}^{(k)}}{1-w_{\hat{j}}^{(k)}}\right)$.
\STATE \textbf{If} $\nabla f(w^{(k)})^\top d_{FW}^{(k)}
\geq \nabla f(w^{(k)})^\top d_A^{(k)}$
\textbf{then} $d^{(k)}:=d_{FW}^{(k)}$ and $\alpha_k^{\max}:=1$;
 \textbf{else} $d^{(k)}:=d_A^{(k)}$ and
$\alpha_k^{\max}:=\dfrac{w_{\hat{j}}^{(k)}}{1-w_{\hat{j}}^{(k)}}$. \label{step:choiceAFWd}
\STATE Choose $\alpha_k\in(0,\alpha_k^{\max}]$ using a suitable criterion.
\STATE Update $w^{(k+1)}:=w^{(k)}+\alpha_kd^{(k)}$.
\STATE Set $k:=k+1$ and go to step~\ref{step:FWdir}.
\end{algorithmic}
\end{algorithm}

% \begin{algorithm}
% \caption{AFW on the standard simplex}
% \label{AFW}
% \begin{algorithmic}[1]
% \STATE Initialize $w^{(0)}\in\Delta_{n}$, $k:=0$. 
% \STATE Set $s^{(k)}:=e_{\hat{i}}$, with $\hat{i}\in\arg\max_{i=1,\dots,n}\nabla_i f(w^{(k)})$ and $d_{FW}^{(k)}:=s^{(k)}-w^{(k)}$.
% \IF{$w^{(k)}$ is stationary} \STATE STOP 
% \ENDIF
% \STATE Let $v^{(k)}:=e_{\hat{j}}$, with $\hat{j}\in\arg\min_{j\in S_k}\nabla_j f(w^{(k)})$, where $S_k:=\{j:w^{(k)}_j>0\}$, and $d_A^{(k)}:=w^{(k)}-v^{(k)}$.
% \IF{$\nabla f(w^{(k)})^\top d_{FW}^{(k)}\geq\nabla f(w^{(k)})^\top  d_A^{(k)}$}
%     \STATE $d^{(k)}:=d_{FW}^{(k)}$,\quad $\alpha_k^{\max}:=1$
% \ELSE
%     \STATE $d^{(k)}:=d_A^{(k)}$,\quad 
%     $\alpha_k^{\max}:=\dfrac{w^{(k)}_{\hat{j}}}{1 - w^{(k)}_{\hat{j}}}$
% \ENDIF
% \STATE Choose the step-size $\alpha_k\in(0,\alpha_k^{\max}]$ with a suitable criterion.
% \STATE Update: $w^{(k+1)}:=w^{(k)}+\alpha_kd^{(k)}$.
% \STATE Set $k:=k+1$. Go to step $2$.
% \end{algorithmic}
% \end{algorithm}

As we can see from Algorithm \ref{AFW}, the AFW method at each iteration chooses between the standard FW direction $d^{(k)}_{FW}$ computed in step \ref{step:FWdir} and the away-step direction $d^{(k)}_A$ calculated in step \ref{step:AFWdir}, with the criterion in step \ref{step:choiceAFWd}. In particular, this second direction shifts the weight away from the worst vertex to the other vertices used to represent the current iterate $w^{(k)}$. For the standard simplex case, this vertex coincides with the canonical basis vector that selects the lowest component of the gradient, or equivalently, the vertex that minimizes the first-order approximation of $f$ given by $z\rightarrow f(w^{(k)}) + \nabla f(w^{(k)})^\top (z-w^{(k)})$ over the support of the current iterate $w^{(k)}$. The upper bound on the step-size $\alpha^{max}_k$ is then the largest feasible step-size for the away direction given the boundary conditions, so, when the algorithm performs an away-step, we have that either the support of the current iterate stays the same or decreases by one. In the latter case, $\alpha_k = \alpha^{max}_k$ and we remove the component whose index is associated with the away direction,  generating a new point on the boundary of the current minimal face. This drop in dimension is what allows the AFW method to quickly identify low-dimensional faces containing $w^{(k)}$, and therefore the correct support of an optimal local solution. On the other hand, when the algorithm performs a standard FW step, it selects the vertex that maximizes the first-order approximation of $f$, which corresponds to the highest component of the gradient, and adds it to the support of the current iterate $w^{(k)}$. These properties are fundamental for the active-set identification of the AFW method. 

Moreover, we often require the following lower bound on the step-sizes:
\begin{align}\label{cond1_stepsize}
    \alpha_k \geq \bar{\alpha}_k := \min \Big( \alpha^{\max}_k , c\frac{\nabla f(w^{(k)})^\top d^{(k)}}{\|d^{(k)}\|_2^2} \Big)
\end{align}
for some $c>0$ and, for convergence results, we need the following sufficient increase condition for some constant $\rho >0$:
\begin{align}\label{cond2_stepsize}
    f(w^{(k)} + \alpha_k d^{(k)}) - f(w^{(k)}) \geq \rho\bar{\alpha}_k\nabla f(w^{(k)})^\top  d^{(k)}
\end{align}
which generalizes properties of exact and Armijo line search. We refer the reader to \cite{AFW} for further details and convergence properties of the AFW algorithm.

At this point, we can introduce our tailored Frank-Wolfe variant for the maximum $s$-plex formulation (\ref{general_regularized_formulation_MsPP}), namely  FW$s$P, as presented in Algorithm \ref{ FW$s$P}, which exploits the product domain structure of the problem at hand by employing separate updating rules for the two blocks of variables $x$ and $y$. 

\begin{algorithm}
\caption{FW$s$P -- FW variant for the $s$-plex}
\label{ FW$s$P}
\begin{algorithmic}[1]
\STATE Initialize $x^{(0)}\in\Delta_n$ and $y^{(0)}\in P'_s(G)$, $k:=0$.
\STATE \textbf{If} $(x^{(k)},y^{(k)})$ is stationary, \textbf{then STOP}. \label{step:statchecka2}
\STATE Compute $x^{(k+1)}$ by applying one iteration of Algorithm~\ref{AFW}
with $w^{(0)}:=x^{(k)}$ and $f(w):=h_G(w,y^{(k)})$. \label{step:AFWa2}
\STATE Let
$y^{(k+1)}\in\arg\max_{y\in P'_s(G)}
\nabla_y h_G(x^{(k+1)},y^{(k)})^\top y$. \label{step:FWa2}
\STATE Set $k:=k+1$ and go to step~\ref{step:statchecka2}.
\end{algorithmic}
\end{algorithm}

% \begin{algorithm}
% \caption{ FW$s$P - FW variant for $s$-plex}
% \label{ FW$s$P}
% \begin{algorithmic}[1]
% \STATE Initialize $x^{(0)}\in\Delta_n$ and $y^{(0)}\in P'_s(G)$, $k := 0$.
% \IF{$(x^{(k)},y^{(k)})$ is stationary}
%     \STATE STOP
% \ENDIF
% \STATE Compute $x^{(k+1)}$ by applying one iteration of Algorithm \ref{AFW} with $w^{(0)} = x^{(k)}$ and $f(w) = h_G(w, y^{(k)})$.
% \STATE Let $y^{(k+1)} \in \arg\max_{y \in P_s'(G)} \nabla_y h_G(x^{(k+1)}, y^{(k)})^\top  y$.
% \STATE Set $k := k + 1$. Go to step $2$.
% \end{algorithmic}
% \end{algorithm}
% \bigskip

In particular, at each iteration, the method first performs one AFW iteration on the standard simplex for the $x$ variable, step \ref{step:AFWa2},  and then a full FW step on the $y$ variable, step \ref{step:FWa2}, which basically means solving a linear optimization problem. Such a problem is the linear relaxation
of a maximum-weight $b$-matching problem on the complement graph
$\bar{G}$, with capacity vector $b=(s-1)e$ and edge-weight vector
$\nabla_y h_G(x^{(k+1)},y^{(k)})$. Throughout the following analysis, we assume that the linear
optimization oracle used in step~\ref{step:FWa2} always returns an extreme-point optimal
solution. This entails no loss of generality, since a linear function
over the nonempty compact polytope $P'_s(G)$ always admits an optimal
solution at an extreme point. Consequently, $y^{(k)}$ is a vertex of $P'_s(G)$ for every $k\geq1$, and, by Lemma \ref{lemma_fractional_b-matching_polytope}, we have that $y^{(k)}_{ij}\in\{0,\frac{1}{2},1\}$ for every $\{i,j\}\in\bar{E}$ . Notice that the order in which these updates are performed is essential, since it allows us to compute the $y^{(k+1)}$ component of the new iterate as a function of the newly computed $x^{(k+1)}$ component, as highlighted in step \ref{step:FWa2} of Algorithm \ref{ FW$s$P}. Furthermore, as stated in the next proposition, the sequence $\{y^{(k)}\}_k$ is ultimately constant. This finally allows us to obtain convergence results by applying the general properties of the AFW method on the standard simplex for the $x$ component, as proved in \cite{AFW}. 

% \IB{Note that the monotonicity assumption on the sequence $h_G(x^{(k+1)}, y^{(k)})$ in the result stated below follows from the sufficient increase condition~\eqref{cond2_stepsize} which is employed in the convergence Theorem~\ref{th_convergence} anyhow.}

\begin{proposition}\label{prop_convergence_y}
    In Algorithm \ref{ FW$s$P}, if $h_G(x^{(k+1)},y^{(k)})
\geq h_G(x^{(k)},y^{(k)})$ for every  $k\geq 0$, 
        then the sequence $\{y^{(k)}\}_k$ converges to a binary vertex of $P'_s(G)$ and can change at most $\frac{2}{\beta(\gamma-1)}\Big(1-\frac{2-\alpha}{2|C^*|} - \frac{\alpha}{2n} + \frac{\beta\bar{m}}{4}(\gamma + 1)\Big)$ times, with an $s$-plex $C^*$ in $G$ of maximum cardinality.
\end{proposition}
\begin{proof}
    First, we show that the function $y\rightarrow h_G(x,y)$ is strongly convex on $[0,1]^{\bar{m}}\supset P'_s(G)$ with a convexity coefficient $\mu = 2\beta(\gamma - 1)>0$, since $\gamma>1$. Since this function is separable with respect to $y$, in the one-dimensional case we can consider $f(t):=w(x)t + \beta t(t-1)(t+\gamma) + c(x)$, $\forall t\in[0,1]$,  for fixed $x$. Although $w(x)$ and $c(x)$ may depend on $x$, they are respectively a coefficient of a linear term and a constant term with respect to $t$; therefore they do not contribute to the second derivative. Indeed, we get $f''(t)=2\beta(3t + \gamma-1)\geq2\beta(\gamma-1)$, since $f''$ reaches its minimum in $[0,1]$ for $t=0$. This shows that $f$ is strongly convex in $[0,1]$ with the strong-convexity coefficient $\mu:=2\beta(\gamma-1)$. Because the Hessian with respect to $y$ of $h_G(x,y)$ is diagonal with entries $f''(y_{ij})$, the same bound applies in $\bar{m}$ dimensions. 
    
    Now assume that $y^{(k)}$ and $y^{(k+1)}$ are two distinct vertices of $P'_s(G)$. Let $z^{(k)}=(x^{(k+1)}, y^{(k)})$ and $z^{(k+1)}=(x^{(k+1)}, y^{(k+1)})$. By the strong convexity just established  for $y\rightarrow h_G(x,y)$, we have
    \begin{align*}
        h_G(z^{(k+1)}) - h_G(z^{(k)})&\geq\nabla h_G(z^{(k)})^\top (z^{(k+1)}-z^{(k)}) + \frac{\mu}{2}\|z^{(k+1)} - z^{(k)}\|^2 \\
        &= \nabla_yh_G(z^{(k)})^\top (y^{(k+1)} - y^{(k)}) + \frac{\mu}{2}\|y^{(k+1)} - y^{(k)}\|^2 \\
        &\geq \frac{\mu}{2}\|y^{(k+1)} - y^{(k)}\|^2\geq \frac{\mu}{2}\cdot\frac{1}{2} = \frac{\mu}{4}. 
    \end{align*}
    The first-order term is non-negative because $y^{(k+1)}$ is an
optimal solution of the linear problem in step~\ref{step:FWa2} of Algorithm \ref{ FW$s$P} and $y^{(k)}$ is feasible. Moreover, the squared distance between two distinct vertices of $P'_s(G)$ is at least $1/2$. Indeed, the vertices are
half-integral, and a single difference of size $1/2$ would contradict the odd-cycle characterization of their fractional support.
    % since the square distance between two distinct vertices of $P'_s(G)$ is at least $\frac{1}{2}$, as they differ in at least two fractional components. 
    % Notice that the scalar product with the gradient is non-negative since we are computing $y^{(k+1)}$ as the vertex that indeed maximizes the considered quantity, see Step $6$ of Algorithm \ref{ FW$s$P}. 
The previous inequality accounts only for the $y$ update. To obtain an increase between two consecutive iterates, we use the assumption  on the $x$ update. As a consequence, whenever  $y^{(k+1)}\neq y^{(k)}$, we have $h_G(x^{(k+1)},y^{(k+1)})
-h_G(x^{(k)},y^{(k)})
\geq \frac{\mu}{4}.$  Hence, these increases can be telescoped over all iterations in which
the $y$-component changes.
    Therefore, $y^{(k)}$ can change at most 
    \begin{align*}
        \max_{(x,y)\in\Delta_n\times P'_s(G)} &\frac{h_G(x,y)-h_G(x^{(0)},y^{(0)})}{\frac{\mu}{4}} \\
        &\leq\frac{4}{\mu}\Big(\max_{(x,y)\in\Delta_n\times P'_s(G)}h_G(x,y)\ \ - \min_{(x,y)\in\Delta_n\times P'_s(G)}h_G(x,y)\Big) \\
        &\leq \frac{4}{\mu}\Big(1-\frac{2-\alpha}{2|C^*|} - \frac{\alpha}{2n} + \frac{\beta\bar{m}}{4}(\gamma + 1)\Big) \\
        &= \frac{2}{\beta(\gamma-1)}\Big(1-\frac{2-\alpha}{2|C^*|} - \frac{\alpha}{2n} + \frac{\beta\bar{m}}{4}(\gamma + 1)\Big)
    \end{align*}
    times, where we used Corollary \ref{corollary_global_equivalence} with an $s$-plex $C^*$ of maximum cardinality in $G$, and the fact that
    \begin{align*}
        \min_{(x,y)\in\Delta_n\times P'_s(G)}h_G(x,y) &\geq \frac{\alpha}{2n} + \min_{y\in P'_s(G)}\beta\sum_{\{i,j\}\in\bar{E}}y_{ij}(y_{ij}-1)(y_{ij}+\gamma)\\
        &\geq\frac{\alpha}{2n} + \beta\bar{m}\min_{t\in[0,1]}t(t-1)(t+\gamma)\\
        &=\frac{\alpha}{2n} -\beta\bar{m}\max_{t\in[0,1]}t(1-t)(t+\gamma)\\
        &\geq\frac{\alpha}{2n} -\beta\bar{m}\max_{t\in[0,1]}t(1-t)\cdot\max_{t\in[0,1]}(t+\gamma)\\
        &=\frac{\alpha}{2n} -\frac{\beta\bar{m}}{4}(\gamma + 1).
    \end{align*}
This is obtained by considering that the adjacency contribution to $h_G$ is non-negative and $||x||_2^2\geq1/n$ for every $x\in\Delta_n$. 
Finally, if after this number of iterations the $y$-component of the solution corresponds to a fractional vertex of $P'_s(G)$, then the algorithm simply sets all fractional components to zero in the subsequent iteration. This occurs because the weights given by the gradient and associated with the fractional components of the solution are strictly negative in this case, since $\beta>2$.
\end{proof}

\begin{remark}
For $\beta>0$ and $\gamma>1$, the strong-convexity argument still
implies finite stabilization of the $y$-sequence, using only the
boundedness of the objective on the compact feasible set. The stronger
condition $\beta>2$ is needed to ensure that the stabilized vertex is
binary and to invoke the equivalence results of Section~\ref{section_3}.
We also observe that the monotonicity assumption in Proposition~\ref{prop_convergence_y} is satisfied under the hypotheses of the convergence result below. Indeed, condition~\eqref{cond2_stepsize}, together with the fact that the direction selected by the AFW method is an ascent direction, implies
$
h_G(x^{(k+1)},y^{(k)})
\geq h_G(x^{(k)},y^{(k)}).
$
Consequently, Proposition~\ref{prop_convergence_y} applies under the assumptions of Theorem~\ref{th_convergence}.

\end{remark}

% \textcolor{blue}{\begin{remark}
%     Notice that for proving the strong convexity we only require $\gamma>1$ and $\beta>0$, so we can relax the $\beta>2$ constraint at least in the numerical experiments and we still have this convergence result for the $y$-component. Indeed the constraint $\beta>2$ serves to be sure that the $y$-component of a local optimal solution is a binary vertex of $P'_s(G)$.
% \end{remark}}

We finally show the convergence of the newly proposed method.
\begin{theorem}\label{th_convergence}
    Let $\{(x^{(k)}, y^{(k)})\}_k$ be the sequence of points generated by Algorithm \ref{ FW$s$P}, with regularization coefficient $\alpha=1$. If conditions \eqref{cond1_stepsize} and \eqref{cond2_stepsize} hold on the step-sizes, then the sequence converges to a first-order stationary point in the
sense of Definition~\ref{def_stationary_point} and identifies the
support of its $x$-component in finite time.
\end{theorem}
\begin{proof} By Proposition~\ref{prop_convergence_y}, after finitely many iterations, the $y$-variable generated by Algorithm~\ref{ FW$s$P} stabilizes at a binary vertex $y^*\in P'_s(G)$. Thereafter, the algorithm reduces to the AFW method applied over the standard simplex to maximize the quadratic function
$
h_G(x,y^*)=x^\top A_{G^*}x+\frac{\alpha}{2}\lVert x\rVert_2^2.
$ As already noticed, the considered optimization problem corresponds to a regularized version of the MS for the MCP over the auxiliary graph $G^*$ induced by the binary vector $y^*$. 
    
   For $\alpha=1$, the fixed-$y$ objective can be written as $h_G(x,y^*)
=
x^\top Q_{G^*}x,$ with 
$Q_{G^*}:=A_{G^*}+\frac{1}{2}\mathbb I_n.
$
% Let $\widehat Q$ be any principal submatrix of $Q_{G^*}$ of order $r$, and let $\widehat A$ be the corresponding principal submatrix of $A_{G^*}$. Then
% \begin{align*}
% 2\widehat Q=2\widehat A+\mathbb I_r.
% \end{align*}
% Since $\widehat A$ has integer entries, the matrix $2\widehat Q$ reduces to $\mathbb I_r$ modulo $2$. Therefore, $\det(2\widehat Q)$ is an odd integer and hence is nonzero. Since
% \begin{align*}
% \det(2\widehat Q)=2^r\det(\widehat Q),
% \end{align*}
% we also have $\det(\widehat Q)\neq0$. 
By the argument used in the proof of Theorem~$12$ in \cite{bomze1999genetic}, every principal submatrix of $Q_{G^*}$ is nonsingular. Consequently, by Theorem~$8$ therein, there can be at most one stationary point in the relative interior of each face of the standard simplex when $y^*$ is fixed. Furthermore, Theorem~$2.5$ in \cite{massaro2002complementary} implies that there exist finitely many stationary points of the fixed-$y$ problem, all satisfying strict complementarity and having distinct supports. Therefore, all the assumptions of Theorem~$4.5$ in \cite{AFW} are satisfied, and the sequence $\{x^{(k)}\}_k$ converges to a stationary point $x^*$ of the fixed-$y$ problem and identifies its support in finite time.

It remains to verify stationarity with respect to the $y$-component. Since $y^{(k+1)}=y^{(k)}=y^*$ for every sufficiently large $k$, step~\ref{step:FWa2} of Algorithm~\ref{ FW$s$P} gives
\begin{align*}
    \nabla_y h_G(x^{(k+1)},y^*)^\top(\bar y-y^*)\leq0
    \qquad\text{for every }\bar y\in P'_s(G).
\end{align*}
Passing to the limit, we obtain
\begin{align*}
    \nabla_y h_G(x^*,y^*)^\top(\bar y-y^*)\leq0
    \qquad\text{for every }\bar y\in P'_s(G).
\end{align*}
Combining this inequality with the stationarity of $x^*$ for the fixed-$y$ problem shows that
$\nabla h_G(x^*,y^*)^\top
    \big((\bar x,\bar y)-(x^*,y^*)\big)\leq0 $
for every $(\bar x,\bar y)\in\Delta_n\times P'_s(G)$. Hence, $(x^*,y^*)$ is a first-order stationary point in the sense of Definition~\ref{def_stationary_point}. Finally, since $y^{(k)}$ becomes constant and the AFW method identifies $\operatorname{supp}(x^*)$ in finite time, the complete sequence converges to $(x^*,y^*)$ and identifies the support of its $x$-component in finite time.
% Thus, every principal submatrix of $Q_{G^*}$ is nonsingular. This is the nonsingularity property used in the proof of Theorem~$12$ in \cite{bomze1999genetic}; consequently, by Theorem~$8$ therein, there can be at most one stationary point in the relative interior of each face of the standard simplex when $y^*$ is fixed.
% For $\alpha = 1$, by the proof of Theorem $12$ in \cite{bomze1999genetic}, we find that all principal minors of the auxiliary induced matrix $A_{G^*}$ do not vanish and, consequently, by Theorem $8$ in \cite{bomze1999genetic}, there can be at most one stationary point in the relative interior of any face of the standard simplex for $y^*$ fixed. 
    % Furthermore, by Theorem $2.5$ in \cite{massaro2002complementary}, the strict complementarity conditions hold in all the stationary points of the complete formulation. In particular, there exist a finite number of stationary points, all satisfying strict complementarity conditions and with distinct supports. Finally, we have all the assumptions to conclude by Theorem $4.5$ in \cite{AFW} and the properties of the AFW algorithm over the standard simplex.
\end{proof}

Although all local maximizers of formulation~\eqref{general_regularized_formulation_MsPP} correspond to $s$-plexes by Theorem~\ref{general_local_equiv_prop}, Theorem~\ref{th_convergence} does not rule out convergence of Algorithm~\ref{ FW$s$P} to saddle points. This possibility, however, does not appear to be problematic in practice. First, in our numerical tests, the method always converged to a local maximizer, in line with studies showing that many first-order methods avoid saddle points with high probability; see, e.g., \cite{bomze1999genetic,lee2019first}. Second, while local maximizers may be locally attractive, saddle points are typically unstable and are unlikely to attract the iterates of a strictly monotone ascent method under standard regularity conditions. Lastly, for our specific problem, there are inexpensive strategies to escape saddle points even when the starting point is ``unlucky,'' for example, when it coincides with a saddle point.

We now describe one such strategy for Algorithm \ref{ FW$s$P}, to be applied after the $y$-component has stabilized at a binary vertex $y^*$, and if the $\operatorname{supp}(x^{(k)})$ is not yet a clique in the auxiliary graph $G^*$.  The first step is to select a pair of nonadjacent vertices
$\{i,j\}\in
\binom{\operatorname{supp}(x^{(k)})}{2}
\setminus\big(E\cup\operatorname{supp}(y^*)\big),
$
an operation that requires checking at most $\binom{|\operatorname{supp}(x^{(k)})|}{2}$ entries of the adjacency matrix of $G^*$.
% The first step is to select an edge $\{i,j\}\in\binom{\operatorname{supp}(x^{(k)})}{2}\setminus E$, an operation that requires checking at most $\binom{\operatorname{supp}(x^{(k)})}{2}$ entries of the adjacency matrix.
Assuming without loss of generality that $\frac{\partial}{\partial x_i}h_G(x^{(k)}, y^{*})\leq\frac{\partial}{\partial x_j}h_G(x^{(k)}, y^{*})$,
the second step is to replace $x_i^{(k)}$ and $x_j^{(k)}$ with $(1-\varepsilon)x_i^{(k)}$ and $x_j^{(k)}+\varepsilon x_i^{(k)}$, respectively, for some fixed $\varepsilon\in(0,1]$. Denoting the resulting point by $\tilde x$ and setting $\delta:=\varepsilon x_i^{(k)}>0$, the quadratic structure of the objective gives
\begin{align*}
h_G(\tilde x,y^*)-h_G(x^{(k)},y^*)
=
\delta\left(
\frac{\partial h_G}{\partial x_j}(x^{(k)},y^*)
-
\frac{\partial h_G}{\partial x_i}(x^{(k)},y^*)
\right)
+\alpha\delta^2>0,
\end{align*}
where the last term follows from the fact that $\{i,j\}$ is a missing edge of $G^*$. The point $(\tilde x,y^*)$ can therefore be used as a new starting point for Algorithm~\ref{ FW$s$P}. In particular, if $(x^{(k)},y^*)$ is a saddle point, a strictly monotone restart cannot converge back to the same saddle point.

% the second step is to replace $x_i^{(k)}$ and $x_j^{(k)}$ with $(1 - \IB{\varepsilon})x_i^{(k)}$ and $x_j^{(k)} + \IB{\varepsilon} x_i^{(k)}$, respectively, for some fixed $\IB{\varepsilon}\in(0,1]$. The resulting point can then be used as a new starting point for Algorithm \ref{ FW$s$P}. It is not difficult to prove that if $(x^{(k)}, y^{(k)})$ is close enough to a saddle point, then the algorithm escapes it after restarting.

Let $y^*$ be a fixed binary vertex of $P'_s(G)$ found by Algorithm \ref{ FW$s$P} after a finite number of iterations, and denote as usual by $G^*$ the auxiliary graph induced by $y^*$. For a clique $C$ of $G^*$ we define $m(C,G^*)=
\min_{v\in V\setminus C} |C| - |N^*(v)\cap C|,
$
that is, the minimum number of missing edges between an
outside vertex $v$ and $C$ that would have to be added to make
$C\cup\{v\}$ a clique.
% the minimum number of edges needed to increase by $1$ the size of the clique. 

We now give an explicit bound on how close the sequence $\{x^{(k)}\}_k$ generated by Algorithm \ref{ FW$s$P} must be to the characteristic vector $x(C)$ for identification to occur.
\begin{proposition}
    Let $\{(x^{(k)}, y^{(k)})\}_k$ be a sequence generated by Algorithm \ref{ FW$s$P}, $y^*$ a binary vertex of $P'_s(G)$ and $C$ a clique in the auxiliary graph $G^*$. Let $\delta_{max}$ be the largest eigenvalue of the the adjacency matrix related to the auxiliary graph $A_{G^*}:=A_G + A_{\bar{G}}(y^*)$. Let $k^*$ be a fixed index in $\mathbb{N}_0$ and $I^c$ be the components of $\operatorname{supp}(x^{(k^*)})$ with index not in $C$, and let $L:=2\delta_{max} + \alpha$. Assume that $y^{(k^*+j)}=y^*$ is constant for $0\leq j\leq |I^c|$, that \eqref{cond1_stepsize} holds for $c=\frac{1}{L}$, and that
    \begin{align}\label{bound}
        \|x^{(k^*)} - x(C)\|_1\leq \frac{m_\alpha(C, G^*)}{m_\alpha(C, G^*) + 2|C|\delta_{max} + |C|\alpha}
    \end{align}
    for $m_\alpha(C, G^*) = m(C, G^*) - 1 + \frac{\alpha}{2}$. Then, the $\operatorname{supp}(x^{(k^* + |I^c|)})=C$.
\end{proposition}
\begin{proof}
    Since after a certain number of iterations the $y$ component of the sequence is fixed, the proof of the above result is identical to the proof reported 
    in~\cite[Proposition~5.4, Appendix]{s-defective}.
\end{proof}
\begin{remark}
    It is well-known that for any graph the largest eigenvalue $\delta_{max}$ of the adjacency matrix is less than or equal to $d_{max}$ as the maximum vertex degree of $G^*$ (see, e.g., \cite{cvetkovic1990largest}). Then condition \eqref{bound} can be replaced by
    \begin{align*}
        \|x^{(k^*)} - x(C)\|_1\leq \frac{m_\alpha(C, G^*)}{m_\alpha(C, G^*) + 2|C|d_{max} + |C|\alpha}.
    \end{align*}
\end{remark}

\section{Numerical results}\label{section_6}
This section presents the results of the computational experiments conducted with the newly proposed  FW$s$P method, Algorithm~\ref{ FW$s$P}. All experiments use the specific objective function introduced in~\eqref{particular_objective}. We also provide a numerical comparison between the original formulation~\eqref{MS_splex_formulation} and its regularized counterpart~\eqref{general_regularized_formulation_MsPP}. We remark that, although the method finds local maxima whose $x$ supports are $s$-plexes rather than necessarily maximum $s$-plexes in practice, it can be used as a heuristic to obtain lower bounds on the $s$-plex number, either independently or within a global optimization scheme. The experiments were conducted with the following objectives:
\begin{itemize}
    \item to show that the proposed FW variant performs better than the general-purpose solver CONOPT when applied to the regularized problems, both in terms of CPU time and the quality of the solutions found; 
    \item to show that our approach, based on solving the regularized formulation~\eqref{general_regularized_formulation_MsPP} with the FW$s$P algorithm, finds solutions of comparable quality to those obtained by applying the general-purpose solver CONOPT directly to the original cubic MS formulation~\eqref{MS_splex_formulation}, while requiring shorter running times and avoiding failures in the identification of maximal $s$-plexes.
\end{itemize}

In the experiments, the regularization parameters were set to $\alpha=1$, $\gamma = 1 + \frac{1}{n}$ and $\beta = \frac{1}{10\cdot \omega(G)^2}$, where $\omega(G)$ is the best-known clique number\footnote{If  the clique number is not  known, this scaling would require replacing $\omega(G)$ with an available estimate.} for the given graph instance~$G$.
% The choice of $\beta$ is an empirical, instance-dependent heuristic intended to moderate the strength of the regularization and improve the practical behavior of the algorithm. We emphasize that this value lies outside the sufficient theoretical regime $\beta>2$ required by condition~\ref{cond_R2}. Consequently, the theoretical guarantees excluding half-integral local maximizers and ensuring a binary missing-edge component do not directly apply to these experimental runs. The reported results should therefore be interpreted as empirical evidence of the practical effectiveness of this parameter choice, rather than as a numerical verification of the theoretical equivalence results.  
An intuitive motivation for this empirical choice of $\beta$ can be obtained by revisiting the argument used in the proof of Theorem~\ref{general_local_equiv_prop}. For a fractional component $y_{ij}^*=1/2$, that argument excludes local optimality whenever $2x_i^*x_j^*+\psi'\left(\frac{1}{2}\right)<0.$
When considering  the regularization function in~\eqref{particular_objective}, we have $\psi'(1/2)=-\beta/4$, and hence the above inequality is equivalent to $\beta>8x_i^*x_j^*.$
The condition $\beta>2$ follows from the uniform bound $x_i^*x_j^*\leq1/4$ over the entire standard simplex. In practice, however, we expect the iterates to approach the normalized characteristic vector of a maximum $s$-plex $C^*$. At such a point, the nonzero products $x_i^*x_j^*$ are equal to $1/|C^*|^2$. Moreover, since every clique is also an $s$-plex, $|C^*|\geq \omega(G)$ and therefore
$\frac{1}{|C^*|^2}\leq\frac{1}{\omega(G)^2}.$
This observation motivates, at a heuristic level, scaling $\beta$ as $\frac{1}{10\cdot \omega(G)^2}$   in our experiments, thereby avoiding an excessively strong regularization effect. Although this argument does not provide the uniform guarantee ensured by $\beta>2$, it offers an intuitive justification for the parameter scaling adopted in the computational experiments.

% An intuitive motivation for this choice of $\beta$, even if it is not greater than $2$ as required in (\ref{particular_objective}), can be given by the fact that we expect the $x$-component of a local maximizer to converge to the characteristic vector of a maximum $s$-plex $C^*$ in the graph $G$, where we know $|C^*|\geq \omega(G)$ since for a parameter $s>1$ we have a proper relaxation of the clique structure. Thus, in the proof of Proposition \ref{general_local_equiv_prop} we could substitute the non-null values of the characteristic vector $x^*$ with $\frac{1}{c\cdot \omega(G)^2}$, for a parameter $c>0$, when showing that we do not have fractional $y$-values for a local maximizer. 

We used the step-size $\alpha_k = \bar{\alpha}_k$ with $\bar{\alpha}_k$ given by (\ref{cond1_stepsize}) for $c=1$, corresponding to an estimate for the Lipschitz constant $L$ of $\nabla_x h_G$. The Short Step Chain (SSC) procedure was adopted in the algorithm, as a gradient recycling scheme to use first-order information more efficiently (see \cite{rinaldi2023avoiding} for details). The code was written in MATLAB, and the tests were performed on an Intel(R) Core(TM) Ultra 7 155H (1.40 GHz)  32 GB. The $50$ graph instances that we used in our experiments are taken from the second DIMACS Implementation Challenge \cite{johnson1996cliques}. These graphs are a common benchmark to assess the performance of algorithms for maximum relaxed clique problems (see \cite{s-defective} and \cite{cubic} and the references therein), and the particular instances we selected coincide with the ones employed in \cite{cubic} to ensure a fair comparison of the performance. Following the scheme adopted in \cite{cubic}, for each pair $(G,s)$, with $G$ a graph from the $50$ selected instances and $s$ an integer in $\{2,3,4,5\}$, we imposed a cumulative time limit of $600$ seconds. Within this time budget, the method was repeatedly run from independently generated random starting points until either $100$ restarts had been completed or the time limit had been reached, whichever occurred first.
% we set a global time limit of $600$ seconds and employed a simple restart scheme with up to $100$ random starting points. 
In particular, the $x$-component of the starting point was generated with MATLAB function \texttt{rand()} and then normalized by dividing it by its sum. A slightly different procedure was adopted for the choice of the starting $y$-component: random values, again generated by the \texttt{rand()} function of MATLAB, were assigned only to the entries of $y$ corresponding to a maximum fractional $(s-1)e$-matching in the complementary graph of the instance. This initialization ensures that the resulting $y$-component is both feasible and randomly generated. To measure approximate stationarity in the numerical experiments, we use the FW gap associated with Definition~\ref{def_stationary_point}. Let $F$ denote the objective
function of the formulation under consideration and let
$\Omega:=\Delta_n\times P'_s(G)$. For any $z=(x,y)\in\Omega$, we define
\begin{align*}
    g_{\mathrm{FW}}^F(z)
    :&=
    \max_{(u,v)\in\Omega}
    \nabla F(x,y)^\top\big((u,v)-(x,y)\big),\\
    &=
    \max_{u\in\Delta_n}
    \nabla_xF(x,y)^\top(u-x)\ 
    +
    \max_{v\in P'_s(G)}
    \nabla_yF(x,y)^\top(v-y).
\end{align*}
The gap is non-negative, and $g_{\mathrm{FW}}^F(z)=0$ if and only if $z$ is stationary in the sense of Definition~\ref{def_stationary_point}. 
Finally, following~\cite{s-defective}, termination of the FW$s$P algorithm required two conditions: the FW gap defined above had to be less than or equal to the prescribed tolerance $\varepsilon:=2\cdot10^{-3}$, and the current support of the $x$-component had to induce an $s$-plex. With this stopping rule, the FW$s$P method satisfied both conditions in all the tested runs, and no support-identification failures were observed. These numerical results are consistent with the local equivalence properties established in this paper.

% \textcolor{red}{Finally, as done in \cite{s-defective}, to terminate the  FW$s$P algorithm two conditions are needed: the presented FW gap is less than or equal to a tolerance set at $\IB{\varepsilon}:=2\cdot10^{-3}$, and the current support of the $x$-component induces an $s$-plex. In particular, with these choices, the experiments show that the  FW$s$P method always terminates after having identified an $s$-plex, thus providing an empirical verification of the local equivalence results we proved in this paper. } 
%In particular, with these choices, the experiments show that the  FW$s$P always terminates after having identified an $s$-plex when applied to formulation (\ref{general_regularized_formulation_MsPP}), while a significant number of failures are observed when applied instead to the original MS formula (\ref{MS_splex_formulation}), thus providing an empirical verification of the local equivalence results we proved in this paper.
For the CONOPT solver the default random starting values and stopping criterion were used. We note that no postprocessing or repair procedure was applied to the solutions returned by either CONOPT or  FW$s$P. This choice was deliberate, since the purpose of the comparison is to assess the intrinsic ability of the original and regularized formulations to produce solutions whose supports directly identify $s$-plexes. Applying a repair procedure could improve the final combinatorial solutions, but would conceal the support-identification failures that the proposed regularization is designed to eliminate.

In the box plots presented below, each series consists of $50$ values corresponding to the aggregated data for the runs performed on the $50$ instances. Here we list the strategies considered in our experiments:
\begin{itemize}
    \item Strategy~$1$ (abbreviated as $S1$) is our proposed approach, in which the regularized formulation~\eqref{general_regularized_formulation_MsPP} is solved using the FW$s$P algorithm (Algorithm~\ref{ FW$s$P}) with the parameters specified above;
    
    \item Strategy~$2$ (abbreviated as $S2$) consists of solving the regularized formulation~\eqref{general_regularized_formulation_MsPP} directly with the general-purpose solver CONOPT;
    
    \item Strategy~$3$ (abbreviated as $S3$) consists of solving the original cubic MS formulation~\eqref{MS_splex_formulation} directly with CONOPT, in line with the approach in~\cite{cubic}.
\end{itemize}
% \textcolor{red}{\begin{itemize}
%     \item Strategy $1$ (abbreviated $S1$) corresponds to our proposed approach to solve the regularized version (\ref{general_regularized_formulation_MsPP}) using the  FW$s$P Algorithm \ref{ FW$s$P} with the parameters reported above;
%     \item Strategy $2$ and Strategy $3$ (abbreviated $S2$ and $S3$) consists of solving the regularized problem (\ref{general_regularized_formulation_MsPP}) and the original one (\ref{MS_splex_formulation}), respectively, directly using the CONOPT solver, following the choice of \cite{cubic}.
% \end{itemize}}
The red lines represent the median of the values in each series, and the boxes extend from the $25$th percentile of the observed data to the $75$th percentile. The whiskers extend to the most extreme observations within the prescribed interquartile range, while values outside this range are displayed as outliers.
% The whiskers cover all other values in a range of $[q1 - w(q3 - q1), q3 + w(q3 - q1)]$, with the coefficient $w$ equal to $2.7$ times the standard deviation of the values. 
% \begin{figure}[ht]
%     \centering
%     \includegraphics[width=1\linewidth]{table_size.eps}
%     \caption{$Si$-$j$ is the box plot related to the average $s$-plex size found for the instance by strategy $Si$ for parameter $s=j$, divided by the known clique number of the instance.}
%     \label{size_results}
% \end{figure}
In Figure \ref{size_results}, the box plot $Si$-$j$ represents the distribution of the average cardinality of the $s$-plex found by Strategy $Si$ with parameter $s=j$, divided by the best-known clique number of the instance. 
To distinguish the algorithms visually, the box plots corresponding to the FW$s$P method ($S1$) are shown in green, whereas those corresponding to CONOPT ($S2$ and $S3$) are shown in blue.
% \textcolor{red}{To highlight the different algorithms tested, we decided to mark the box-plots for the  FW$s$P method in green, $S1$, while the blue was assigned to the boxes of CONOPT solver, $S2$ and $S3$.}
Notice that some normalized values exceed $1$. This is expected because, for $s>1$, the $s$-plex model is a relaxation of the clique model, and an $s$-plex may therefore have cardinality larger than the clique number. The solutions obtained by applying the FW$s$P algorithm to the regularized formulation~\eqref{general_regularized_formulation_MsPP} ($S1$) are generally larger than those returned by CONOPT for the same formulation ($S2$). Moreover,  the median solution sizes obtained by $S1$ and by CONOPT applied to the original cubic MS formulation~\eqref{MS_splex_formulation} ($S3$) are comparable.

% Notice that some data points are greater than $1$, as expected since for $s>1$ the cardinality of an $s$-plex can exceed the clique number as a proper relaxation of the original clique model. \textcolor{red}{The solutions obtained using the  FW$s$P algorithm for the regularized problem (\ref{general_regularized_formulation_MsPP}), $S1$, are generally better than the ones obtained using the CONOPT solver on the same regularized problem, $S2$. Furthermore, while the variance is higher, no significant differences can be seen in the median quality of the solutions found by $S1$ and CONOPT applied to the MS cubic formulation (\ref{MS_splex_formulation}), $S3$.}
%On the other hand, the quality of the solutions obtained using the CONOPT solver combined with a post-processing algorithm, $S4$, is slightly better than that achieved by the other strategies. 
% \begin{figure}[ht]
%     \centering
%     \includegraphics[width=1\linewidth]{table_time.eps}
%     \caption{$Si$-$j$ is the box plot related to the average running time (in seconds) for strategy $Si$ with parameter $s=j$.}
%     \label{time_results}
% \end{figure}

\begin{figure}[ht]
    \centering

    \begin{subfigure}[t]{0.49\linewidth}
        \centering
        \includegraphics[width=\linewidth]{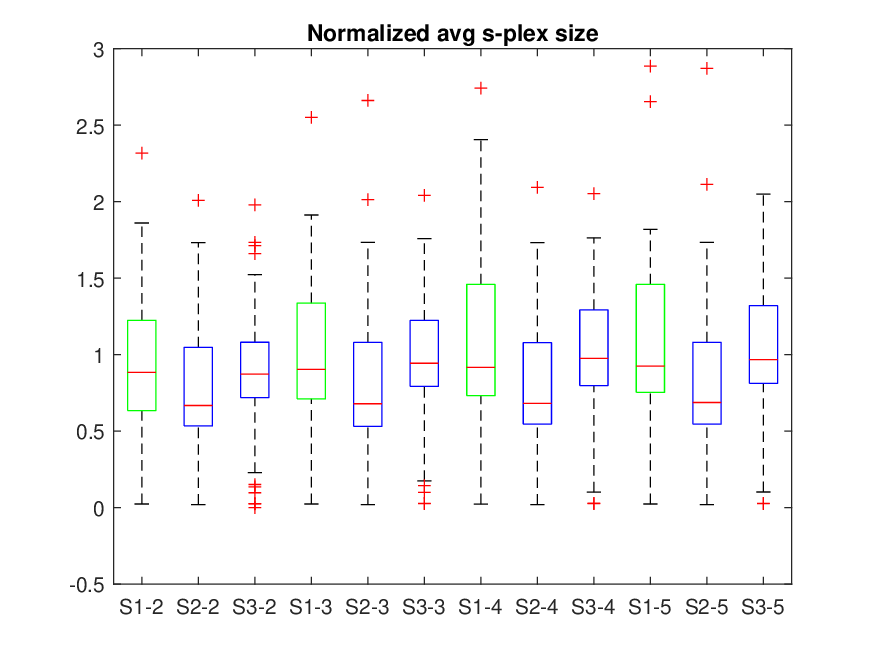}
        \caption{The box plot $Si$-$j$ reports the average $s$-plex size
        found by strategy $Si$ for $s=j$, divided by the best-known clique
        number of the instance.}
        \label{size_results}
    \end{subfigure}
    \hfill
    \begin{subfigure}[t]{0.49\linewidth}
        \centering
        \includegraphics[width=\linewidth]{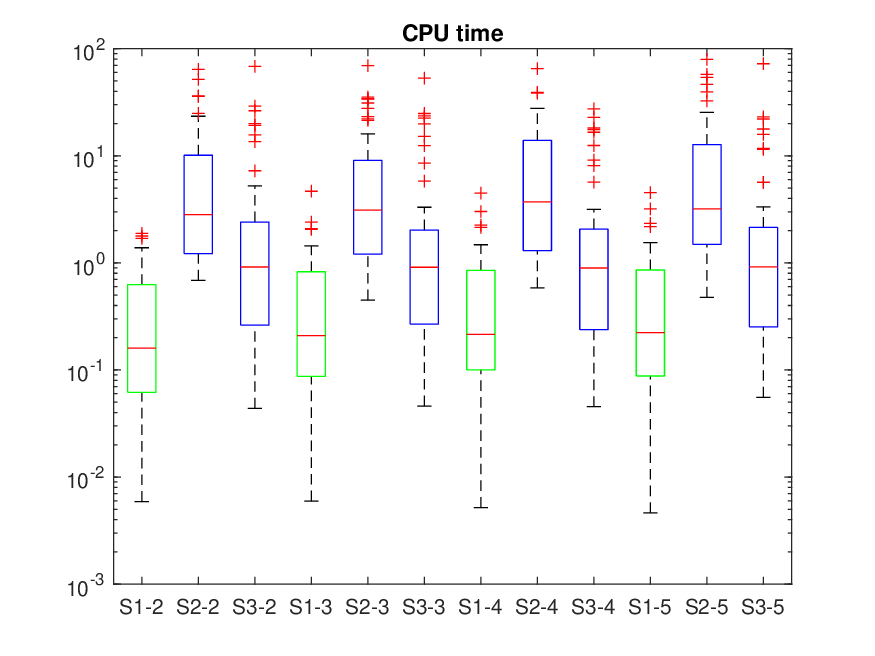}
        \caption{The box plot $Si$-$j$ reports the average running time
        (in seconds) of strategy $Si$ for $s=j$.}
        \label{time_results}
    \end{subfigure}

    \caption{Comparison of the average $s$-plex size and running time
    for the considered strategies.}
    \label{fig:size_time_results}
\end{figure}

In Figure \ref{time_results}, $Si$-$j$ represents the distribution of average CPU  times in seconds (on a logarithmic scale, explaining the asymmetry of the box plots) of Strategy $Si$ for parameter $s=j$.
As in the previous figure, green denotes the FW$s$P method, whereas blue denotes the CONOPT solver. Strategy~$S1$ outperforms $S2$ and $S3$ in terms of computational efficiency, exhibiting a substantially lower median execution time. Moreover, $S1$ completes all $100$ runs within the prescribed time limit of $600$ seconds for every tested instance, whereas $S2$ and $S3$ fail to reach $100$ runs within this time limit for some instances.
% As done previously, we highlight the different algorithms using green for the FW$s$p method, and blue for CONOPT solver. Here we can see that strategy $S1$ outperforms the others, as it requires a much smaller median execution time. \textcolor{red}{Furthermore, the detailed results reported in Appendix \ref{appendix} show that this strategy is always able to complete the $100$ runs within the imposed time limit of $600$ seconds, while, this condition is not met for some instances examined using strategies $S2$ and $S3$.}
%Furthermore, notice that, given the presence of the post processing routines, strategy $S4$ turns out to be less robust than the others, since the variance of the CPU time is larger compared to the others boxes.
Summarizing, the reported results indicate that the proposed Strategy $S1$, based on the regularized version (\ref{general_regularized_formulation_MsPP}) and a tailored FW variant,  is a viable alternative when searching for maximal $s$-plexes in a graph. We refer the reader to the GitHub repository  for detailed numerical results. 

\section{Conclusion}
\label{section_7}
We introduced a general continuous regularization framework for the maximum $s$-plex problem, designed to eliminate the spurious maximizers admitted by the cubic Motzkin-Straus extension \cite{cubic}. By exploiting the two-block structure of the formulation (i.e., the vertex variables and the missing-edge variables) together with the half-integral, odd-cycle structure of the fractional $b$-matching polytope, we derived sufficient conditions under which every local maximizer corresponds to an $s$-plex and every global maximizer yields a maximum $s$-plex. We also presented several families of regularization functions satisfying these conditions.

To solve the resulting formulation, we developed a tailored block-coordinate first-order algorithm that combines away-step Frank-Wolfe updates on the standard simplex with linear optimization over the fractional $b$-matching polytope. Under the stated assumptions, we established finite stabilization of the missing-edge variables, convergence to first-order stationary points, and finite identification of the support of the vertex variables. Experiments on standard benchmark graphs indicate that the tested regularized formulation avoids the support-identification failures observed for the original cubic model and returns high-quality $s$-plexes at a low computational cost.

Several directions merit further investigation. A first question is whether the theoretical guarantees can be retained under weaker or data-dependent regularization conditions, leading to adaptive parameter-selection rules. It would also be valuable to extend the framework to other clique-relaxation and network-clustering models. From a global-optimization perspective, the proposed method could be embedded in multistart, decomposition, or branch-and-bound schemes, where its fast local solutions could provide strong lower bounds. Finally, recent exact continuous formulations for broader classes of graph problems \cite{Butenko2026Continuous} suggest opportunities for combining structure-aware regularization with scalable and parallel first-order algorithms.
\newpage

\bibliographystyle{siamplain}
\bibliography{references_new}

\end{document}